\documentclass[a4paper,11pt]{article}
\usepackage{mathtools}
\usepackage{microtype}
\usepackage[dvipsnames]{xcolor}
\usepackage{amsthm}
\usepackage{amssymb}
\usepackage{comment}
\usepackage{authblk}
\usepackage[active]{srcltx}
\usepackage{tikz}
\usepackage[normalem]{ulem}
\usepackage{csquotes}
\usepackage[hidelinks]{hyperref}

\usepackage{geometry}
\newtheorem{definition}{Definition}[section]
\newtheorem{lemma}[definition]{Lemma}
\newtheorem{proposition}[definition]{Proposition}
\newtheorem{theorem}[definition]{Theorem}

\theoremstyle{definition}
\newtheorem{remark}[definition]{Remark}
\newtheorem{notation}[definition]{Notation}

\numberwithin{equation}{section}

\newcommand{\dd}{\mathop{}\!{\mathrm {d}}}
\newcommand{\ii}{\mathrm {i}}
\newcommand{\ee}{\mathrm {e}}
\newcommand{\weakto}{\rightharpoonup}

\DeclareMathOperator{\curl}{curl}
\DeclareMathOperator{\re}{Re}
\DeclareMathOperator{\Div}{div}
\DeclareMathOperator{\supp}{supp}
\DeclareMathOperator{\dist}{dist}
\DeclareMathOperator{\Span}{Span}

\DeclareMathOperator{\ordo}{o}
\DeclareMathOperator{\Ordo}{\mathcal{O}}

\DeclarePairedDelimiter{\abs}         {\lvert}{\rvert}
\DeclarePairedDelimiter{\measure}     {\lvert}{\rvert}
\DeclarePairedDelimiter{\norm}        {\lVert}{\rVert}
\DeclarePairedDelimiter{\innerproduct}{\langle}{\rangle}
\DeclarePairedDelimiter{\integerpart} {\lfloor}{\rfloor}
\DeclarePairedDelimiter{\set}         {\lbrace}{\rbrace}

\newcommand{\Domain}{U}
\newcommand{\Inner} {V}
\newcommand{\Outer} {W}

\newcommand{\closure}{\overline}
\newcommand{\conjugate}{\overline}
\newcommand{\Hamiltonian}{\mathcal {H}}
\newcommand{\Quadraticform}{\mathfrak {h}}
\newcommand{\Unitary}{\mathcal {U}}
\newcommand{\OuterB}{\sigma}
\newcommand{\Flux}{\Phi}

\newcommand{\reals}         {\mathbb {R}}
\newcommand{\complexes}     {\mathbb {C}}
\newcommand{\integers}      {\mathbb {Z}}
\newcommand{\naturalnumbers}{\mathbb {N}}

\newcommand{\Ab}{\mathbf {A}}
\newcommand{\Fb}{\mathbf {F}}
\newcommand{\Rb}{\mathbf {R}}

\begin{document}
\title {High Flux Asymptotics and Critical Phenomena\\ for the Magnetic Laplacian}

\author[1]{Emanuela L. Giacomelli\thanks{emanuela.giacomelli@unimi.it}}
\affil[1]{University of Milan, Department of Mathematics, Via Cesare Saldini 50, 20133 Milan, Italy}

\author[2]{Ayman Kachmar\thanks{ak292@aub.edu.lb}}
\affil[2]{Department of Mathematics {\normalfont and} PDE Research Unit–Center for Advanced Mathematical Sciences (CAMS) American University of Beirut,
  P.O. Box 11-0236 Riad El-Solh / Beirut 1107 2020 Lebanon}

\author[3]{Mikael Sundqvist\thanks{mikael.persson\_sundqvist@math.lth.se}}
\affil[3]{Department of Mathematics, Lund University, Sweden}

\maketitle
\abstract{%
  We study the lowest eigenvalue of the Neumann magnetic Laplacian in a planar
  domain divided into two regions, with piecewise constant magnetic fields that
  may scale differently in the inner and outer parts. Our aim is to describe the
  high-flux limit and determine when the ground-state energy is eventually
  monotone and when it continues to oscillate.

  We identify several asymptotic regimes according to the relative strength of
  the outer field. When the outer field is fixed, the lowest eigenvalue exhibits
  persistent oscillations and the low-energy states localize in the outer region.
  When the outer field grows more slowly, the behavior depends strongly on the
  geometry: it is eventually monotone for non-circular domains, while
  oscillations may persist for disks. In the critical regime, where the two
  fields are comparable, geometry and flux distribution both play a decisive
  role. When the outer field dominates, the problem reduces asymptotically to an
  effective operator on the inner region.

  These results show how uneven magnetic scaling, topology, and geometry shape
  the high-flux spectral behavior.\par}

\tableofcontents

\section{Introduction}\label{sec:intro}

Oscillations of the lowest eigenvalue of the magnetic Laplacian with respect to
the strength of the magnetic field are closely related to phase transitions in
superconductivity and the celebrated Little--Parks effect~\cite{LP}. In thin
cylindrical superconducting samples, multiple transitions between superconducting
and normal states occur as the magnetic flux varies, suggesting a topological
origin of this non-monotonic behavior. While such oscillations have been
observed in general thin domains~\cite{RS, ShSt} and persist up to very large
magnetic fields~\cite{HK}, they may also arise from magnetic field
inhomogeneities~\cite{FS, KP, KS} or surface effects~\cite{KS}. In contrast,
monotonicity for large fields (strong diamagnetism) typically holds under
homogeneous magnetic fields~\cite[Section~8.5]{FH-b}, and it has been shown to
extend to certain inhomogeneous, sign-changing magnetic fields, notably magnetic
steps~\cite{A}.

In this work we study whether oscillations of the lowest eigenvalue persist in
the presence of a strong magnetic field localized in the inner part of the
domain. Specifically, we let the magnetic intensity $b$ in the inner region tend
to $+\infty$, while allowing the outer intensity \(\OuterB(b)\) to scale
differently. The central question is whether the lowest eigenvalue eventually
becomes monotone, or continues to exhibit oscillatory behavior.

To be more precise, we consider a domain \(\Domain \subset \reals ^ 2\) that is
divided into an inner part \(\Inner\) and an outer part \(\Outer\), and apply the
parameter-dependent magnetic field \(B\) that equals \(b\) in the inner part
\(\Inner\) and \(\OuterB(b)\) in an exterior part \(\Outer\),
\begin{equation}\label{eq:magneticfield}
  B (x)
  =
  \begin{cases}
    b      & x \in \Inner,\\
    \OuterB(b) & x \in \Outer.
  \end{cases}
\end{equation}
We will also assume that \(\OuterB(b)\) is non-decreasing as a function of \(b\).
The picture we have in mind is given in Figure~\ref{fig:domains}. The domains
\(\Domain = \closure {\Inner} \cup \Outer\) as well as \(\Inner\) are supposed to
have smooth boundaries and to be simply connected.

\begin{figure}[htb]
    \centering
    \begin{tikzpicture}[scale= 0.7]
      \draw[preaction={draw}]
        plot[smooth cycle]
        coordinates{
        (-2,5) (0.8,6.5) (2.5,6.9) (3.5,6.5) (5, 6.5)
        (6,4) (6,2) (4,2) (2,1) (0, 1.3) (-2,2)
        } ;
      \draw[preaction={draw}]
        plot[smooth cycle]
        coordinates{
        (-1,4) (-0.5, 3) (1, 2) (2.5, 3) (3,5.5)
        } ;
      \node at (7, 2)   {\(\Domain \)} ;
      \node at (1, 3.5) {\(\Inner\)} ;
      \node at (4, 4)   {\(\Outer\)} ;
    \end{tikzpicture}
    \caption
      {Illustration of the domains \(\Domain, \Inner \subset \reals ^ 2\) with
      \(\closure{\Inner}\subset\Domain\) and \(\Outer = \Domain \setminus \closure
      {\Inner}\).}
    \label{fig:domains}
\end{figure}

For \(b \in \reals\) we denote by \(\Hamiltonian(b)\) the Neumann realization of
the magnetic Laplacian in \(L ^ 2(\Domain)\) associated with the magnetic field
\(B\). It is known to have compact resolvent and therefore a discrete spectrum.
We denote the corresponding increasing sequence of eigenvalues by \(\lambda _
1(b) \leq \lambda _ 2(b) \leq \ldots \). We are mainly interested in whether
\(\lambda_1(b)\) is monotone for large \(b\).

In what follows, we say that:
\begin{itemize}
    \item The map \(b \mapsto \lambda_1(b)\) is \emph{strongly non-diamagnetic} if it is not monotone on any interval of the form \((b_0,\infty)\), for any \(b_0 \geq 0\).
    \item The map \(b \mapsto \lambda_1(b)\) is \emph{strongly diamagnetic} if there exists \(b_0 > 0\) such that it is monotonically increasing for all \(b \in (b_0,\infty)\).
\end{itemize}

Our main results are as follows:
\begin{itemize}
  \item[(i)]
    \emph {Perfect oscillatory regime.}
    If the field strength \(\OuterB(b)\) in \(\Outer\) is constant, then
    \(\lambda_1(b)\) is strongly non-diamagnetic
    (Theorem~\ref{thm:mainoscillatory}).
  \item[(ii)]
    \emph {Subcritical regime.}
    If \(1 \ll \OuterB(b) \ll b\) as \(b \to +\infty\) and \(\Domain\) is not a
    disk, then \(\lambda_1(b)\) is strongly diamagnetic. If \(\Domain\) is a
    disk, then \(\lambda_1(b)\) is strongly non-diamagnetic
    (Theorem~\ref{thm:mainsubcritical}).
  \item[(iii)]
    \emph {Critical regime.}
    If \(\OuterB(b) = a b\), with \(a\) constant, and \(\Domain\) is a disk, then
    \(\lambda_1(b)\) exhibits both strongly diamagnetic and strongly
    non-diamagnetic behavior, depending on the value of \(a\). If \(\Domain\) is
    not a disk, then \(\lambda_1(b)\) is strongly diamagnetic.
  \item[(iv)]
    \emph {Supercritical regime.}
    If \(\OuterB(b) \gg b\) as \(b \to +\infty\), then \(\lambda_1(b)\) behaves
    like the lowest eigenvalue of the operator in \(\Inner\) with Dirichlet
    boundary conditions on \(\partial \Inner\). In this case, \(\lambda_1(b)\) is
    strongly diamagnetic, independently of the geometry of \(\Domain\).
\end{itemize}

The behavior of \(\lambda_1(b)\) is closely related to the magnetic flux through
the domain. In particular, the interplay between the flux contributions from
\(\Inner\) and \(\Outer\) will be a key mechanism behind the different regimes
described above. The flux \(\Flux\) of the magnetic field in the domain
\(\Domain\) is defined by
\[
  \Flux
  =
  \frac {1}{2\pi} \int _ {\Domain} B(x) \dd x
  =
  \frac{1}{2\pi}b\measure {\Inner} + \frac{1}{2\pi}\sigma(b)\measure {\Outer}.
\]
Since \(\sigma(b)\) is non-decreasing in \(b\), the limit \(b \to +\infty\) is
equivalent to the limit \(\Flux \to +\infty\), except possibly in the critical
regime (\(a<0\)). Notably, however, the flux can approach infinity while the
magnetic field scales differently in various regions of \(\Domain\), leading to a
rich variety of behaviors. We quantify this behavior for the lowest eigenvalue.

Our analysis also advances the understanding of spectral asymptotics for step
magnetic fields~\cite{AHK, FHKR, GV, G2}, with particular attention to
configurations featuring uneven scaling across the domain. Focusing on the case
where the magnetic field is strongly localized in an interior region while
remaining fixed elsewhere, we further investigate the distribution of higher
energy levels, extending the analysis beyond the ground-state energy.

Although we consider planar domains, our problem is naturally associated with a
cylindrical geometry under a constant magnetic field parallel to the cylinder's
axis. The study of general three-dimensional domains is of independent interest,
and additional geometric difficulties are expected (see \cite{AG}). Recently,
\cite{Pan} investigated the role of topology in three dimensions in the context
of superconductivity and the Little-Parks effect. Building on these results, the
study of uneven magnetic field scaling in three dimensions is a natural
continuation of our work, potentially extending the results in \cite{AG} and
providing new insights into the Little-Parks effect.

The remainder of this article is organized as follows. Section~\ref{sec:setup}
introduces the geometric setting and defines the magnetic Neumann Laplacian.
Section~\ref{sec:mainresults} presents the main results across the four
asymptotic regimes. Sections~\ref{sec:oscillatory}-\ref{sec:supercritical} are
devoted to the proofs: the perfect oscillatory regime
(Section~\ref{sec:oscillatory}), the subcritical regime
(Section~\ref{sec:subcritical}), the critical regime
(Section~\ref{sec:critical}), and the supercritical regime
(Section~\ref{sec:supercritical}). Two appendices collect technical proofs
(Appendix~\ref{app:proofs}) and the necessary model operators
(Appendix~\ref{sec:pre}).

\section{The setup}\label{sec:setup}

Let \(\Domain,\Inner \subset \reals ^ 2\) be bounded and simply connected domains
with smooth boundaries. We assume that \(\closure {\Inner} \subset \Domain\), and
set \(\Outer = \Domain \setminus \closure {\Inner}\). We refer again to
Figure~\ref{fig:domains} for the picture we have in mind, and we will call
\(\Inner\) the inner domain and \(\Outer\) the outer domain.

For each \(b > 0\) we let the magnetic field \(B\) be defined
by~\eqref{eq:magneticfield}. We will consider different choices of the field
strength \(\OuterB(b)\) in the outer domain; in particular we will let it be
constant or let it grow as a power of \(b\).

We denote by \(\Hamiltonian(b)\) the self-adjoint Neumann realization of the
magnetic Schrödinger operator, acting in \(L ^ 2(\Domain)\). Its action is given
by
\[
  (-\ii\nabla - \Ab) ^ 2
  =
  -\Delta
  + 2\ii \Ab \cdot \nabla
  + \ii \nabla \cdot \Ab
  + \abs {\Ab} ^ 2.
\]
Here \(\Ab\) is a (\(b\)-dependent) magnetic vector potential that generates the
magnetic field \(B = \curl \Ab\). There is freedom of choice of \(\Ab\); we get a
fixed one if we take \(\Ab = (-\partial _ 2 \phi, \partial _ 1 \phi)\) where
\(\phi\) is a scalar potential, which means that \(-\Delta \phi = B\) in
\(\Domain\) and \(\phi = 0\) on the boundary \(\partial \Domain\). This choice
has the advantageous property of being divergence free in \(\Domain\) and
orthogonal to every normal vector on the boundary \(\partial\Domain\). In fact,
we will repeatedly use the fact that such \(\Ab\) can, by superposition, be
written as a sum of two terms, responsible for the flux through the inner and
outer domains:
\begin{equation}\label{eq:Aassum}
  \Ab
  =
  b\Fb _ {\Inner} + \OuterB(b)\Fb _ {\Outer},
\end{equation}
with
\begin{equation}\label{eq:def-F}
  \Fb_{\Inner} = (-\partial_2\phi_{\Inner},\partial_1\phi_{\Inner}),
  \quad
  \Fb_{\Outer} = (-\partial_2\phi_{\Outer},\partial_1\phi_{\Outer}),
\end{equation}
where \(\phi_{\Inner},\phi_{\Outer}\) are respectively the unique solutions of
\[
  \begin{cases}
     -\Delta\phi_{\Inner} = \mathbf{1}_\Inner &\text{ in } \Domain,\\
     \phi _ {\Inner} = 0                      &\text{ on } \partial\Domain,
  \end{cases}
  \qquad
  \begin{cases}
    -\Delta\phi _ {\Outer} = \mathbf{1}_{\Outer} &\text{ in } \Domain,\\
    \phi _ {\Outer} = 0                          &\text{ on } \partial\Domain.
  \end{cases}
\]
Notice that by elliptic regularity and Sobolev inequality \(\Fb_{\Inner}\) and
\(\Fb_{\Outer}\) are continuous and satisfy
\begin{align}\label{eq:prop-A}
  &  \curl\Fb_{\Inner} = \mathbf{1}_{\Inner},
  && \Div\Fb_{\Inner} = 0 \text{ in } \Domain,
  && \nu\cdot\Fb_{\Inner} = 0 \text{ on } \partial\Domain,\nonumber \\
  &  \curl\Fb_{\Outer} = \mathbf{1}_{\Outer},
  && \Div\Fb_{\Outer} = 0 \text{ in } \Domain,
  && \nu\cdot\Fb_{\Outer} = 0 \text{ on } \partial\Domain.
\end{align}
The operator \(\Hamiltonian(b)\) can be defined via the quadratic form
\begin{equation}\label{eq:def-qf}
  \Quadraticform [\psi]
  \coloneqq
  \int _ {\Domain} \abs {(-\ii\nabla - \Ab)\psi} ^ 2 \dd x,
\end{equation}
with form domain given by the magnetic Sobolev space of order one, namely the set
of all \(\psi \in L^2(\Domain)\) such that \(\Quadraticform[\psi] < +\infty\). In
the present setting, this domain coincides with the standard Sobolev space
\(H^1(\Domain)\).

Turning to the domain of the operator, we impose magnetic Neumann boundary
conditions on the boundary $\partial \Domain$, namely
\[
    \nu\cdot (-\ii\nabla - \Ab)u = 0\quad \text{on } \partial \Domain.
\]
Since $\Ab$ is divergence free and satisfies $\nu\cdot \Ab = 0$ on
$\partial \Domain$, this condition reduces to the standard Neumann boundary condition
$\nu\cdot \nabla u = 0$ on $\partial \Domain$. Consequently,
\begin{equation}\label{eq: def domain Hb}
    \mathcal{D}(\Hamiltonian(b))
    =
    \set{u\in H^2(\Domain) \mid \nu\cdot \nabla u = 0\quad \text{on } \partial \Domain},
\end{equation}
and in particular the operator domain is independent of $b$.

As mentioned above, \(\Hamiltonian(b)\) has a compact resolvent, so we can
enumerate its eigenvalues in increasing order, \(\lambda _ 1(b) \leq \lambda _
2(b) \leq \ldots\). These eigenvalues can also be characterized by the min-max
principle,
\begin{equation}\label{eq:def-lambda-n}
  \lambda_n(b)
  =
  \inf\biggl\{\max_{\substack{\psi \in M\\ \psi \neq 0}}
      \frac{\Quadraticform[\psi]}
           {\norm {\psi} ^ 2}
           \colon
           M\subset H^1(\Domain),~\dim(M)=n \biggr\}.
\end{equation}

We will also consider the same quantities introduced above, but restricted to the
outer domain \(\Outer\), while keeping the same vector potential \(\Ab\). When
necessary, we indicate this by a superscript. For instance, we write
\(\Quadraticform ^ \Outer\) for the corresponding quadratic form,
\(\Hamiltonian^{\Outer}(b)\) for the corresponding operator and
\(\lambda_n^{\Outer}(b)\) for its eigenvalues. For the operator
\(\Hamiltonian^{\Outer}(b)\), we impose Dirichlet boundary conditions on the
portion of \(\partial \Outer\) that coincides with \(\partial \Inner\). The
operator is defined via its associated quadratic form, whose domain consists of
functions in the Sobolev space \(H^1(\Outer)\) with vanishing trace on the inner
boundary \(\partial \Inner\). Similar min-max characterizations hold in all the
cases we consider, with the appropriate modifications of the form domain.

\section {Main results}\label{sec:mainresults}

Throughout, we consider the domains \(\Domain,\Inner,\Outer\) to be fixed and
satisfying the properties described earlier. We apply the magnetic field \(B\)
from~\eqref{eq:magneticfield}, and vary the intensity \(\OuterB\) in the outer
domain \(\Outer\) to explore different regimes. We present our results separately
for each regime.

We emphasize that different choices of \(\OuterB\) dramatically change the
behavior of the lowest eigenvalue \(\lambda _ 1(b)\). We also highlight the role
of the geometry of the domain \(\Domain\), which becomes particularly significant
in the subcritical regime.

\subsection{The perfect oscillatory regime}\label{subsec:oscillatory}

In this regime, the magnetic field \(\OuterB(b)\) is constant in the outer region
\(\Outer\). For large \(b\), the magnetic flux through the inner domain
\(\Inner\) dominates the total flux, and, similarly to the case of an
Aharonov--Bohm solenoid~\cite{KP}, induces oscillatory behavior. By perfect
oscillations we mean that these oscillations persist for arbitrarily large values
of \(b\), due to an exact flux-periodicity mechanism.

\begin{theorem}\label{thm:mainoscillatory}
  If \(\OuterB(b) = a\), where \(a \in \reals\) is a fixed constant, then the map
  \(b \mapsto \lambda_1(b)\) is strongly non-diamagnetic.
\end{theorem}

The case $a=0$ was previously obtained in~\cite[Theorem~1.4]{KS1}. We establish
the non-monotonicity by showing that (see Proposition~\ref{prop:asyequi}), for
large \(b\),
\[
  \lambda _ 1 (b) \sim \lambda _ 1 ^ {\Outer}(b),
\]
and by observing that the map \(b \mapsto \lambda _ 1 ^ {\Outer}(b)\) is periodic
in \(b\) (Lemma~\ref{lem:HHOO}, see also~\cite{HHOO}). This periodicity is a pure
flux effect, arising from the fact that integer fluxes can be removed by gauge
transformations, as in the Aharonov--Bohm case. The asymptotic equivalence
suggests that the ground eigenfunctions concentrate in the outer domain
\(\Outer\) as \(b\to\infty\). We show that this is indeed the case
(Lemma~\ref{lem:concentration}).

The asymptotic equivalence extends to higher eigenvalues,
\[
  \lambda_n(b) \sim \lambda_n^{\Outer}(b), \qquad \text{for all } n,
\]
and the corresponding concentration properties of the eigenfunctions remain valid
(see {Propositions~\ref{prop:asyequi}, \ref{prop:conv-ef} and
\ref{prop:ev-high-N}}).

\subsection{The subcritical regime}\label{subsec:subcritical}

We consider here intensities \(\OuterB(b)\) that diverge as \(b\to\infty\), but
at a slower rate than the intensity \(b\) in the inner part of the domain. For
simplicity, we take \(\OuterB(b) = b ^ {\varrho}\), with \(0 < \varrho < 1\).

\begin{theorem}\label{thm:mainsubcritical}
  If \(\OuterB(b) = b ^ {\varrho}\), where \(0 < \varrho < 1\) is a fixed constant, then two distinct behaviors occur, depending on the geometry of the domain \(\Domain\):
  \begin{enumerate}
   \item
      If \(\Domain\) is a disk, then the map \(b \mapsto \lambda _ 1(b)\) is strongly non-diamagnetic.
    \item
      If \(\Domain\) is not a disk,
      then the map \(b
      \mapsto \lambda _ 1(b)\) is strongly diamagnetic.
  \end{enumerate}
\end{theorem}

The exceptional role of the disk may at first seem surprising. The key point is
that, in a general domain, the boundary layer mechanism in \(\Outer\) selects a
preferred point on \(\partial\Domain\), leading to a stabilization of the ground
state energy.

In contrast, rotational symmetry prevents the selection of a unique minimizer
along the boundary of the disk. Consequently, flux effects originating from the
inner region are not suppressed and can induce oscillations in the lowest
eigenvalue, as observed in related radial settings (see \cite{FS}), where the
magnetic field is proportional to a fixed profile and radial, with a minimum at
the boundary.

\subsection{The critical regime}\label{subsec:critical}

When the magnetic field strengths are of the same order, the situation
becomes more subtle. For \(a \in \reals\), we consider the field strength
\(\OuterB(b) = ab\) in \(\Outer\). The case \(a=0\) reduces to the perfect
oscillatory regime \(\OuterB(b)=0\), so we assume \(a\neq 0\).

If \(\Domain\) is not a disk, we again expect (and indeed obtain) monotonicity of
\(\lambda _ 1(b)\) for large \(b\). When \(\Domain\) is a disk, the situation is
more intricate. It is known (see~\cite{FH-b}) that \(\lambda _ 1(b)\) is monotone
increasing for large \(b\) if \(a = 1\), i.e.\ when the strength is uniform
across the entire domain \(\Domain\).

In this regime, the spectral constants $\Theta_0 \approx 0.59$ and
$\varphi_{0}(0)\approx 0.873$, where $\varphi_{0}$ denotes the
$L^2$-normalized ground state of the harmonic oscillator on the half-line
(see Appendix~\ref{sec:dG}), together with the flux
\[
  \Phi_0=\frac{a \measure{\Outer}}{2\pi}
  +\frac{\measure{\Inner}}{2\pi},
\]
play a decisive role, especially in the case of the disk.

\begin{theorem}\label{thm:maincritical}
  Suppose that \(\OuterB(b) = ab\), where \(a\not=0\) is fixed. The diamagnetic behavior depends on the parameter $a$, and in certain cases, on the geometry of the domain:
  \begin{enumerate}
  \item
    If \(-1<a<\Theta_0^{-1}\) and \(\Domain\) is the unit disk, then
    \begin{enumerate}
      \item
        the map \(b \mapsto \lambda _ 1(b)\) is strongly diamagnetic when
        $\Phi_0< a\sqrt {\Theta _ 0}/\varphi_{0}(0)^2$,
      \item
        the map \(b \mapsto \lambda _ 1(b)\) is strongly non-diamagnetic when
        $\Phi_0> a\sqrt {\Theta _ 0}/\varphi_{0}(0)^2$.
    \end{enumerate}
      \item
    If \(-1<a<\Theta_0^{-1}\) and \(\Domain\) is not a disk, then the map \(b
    \mapsto \lambda _ 1(b)\) is strongly diamagnetic.
  \item
    If \(a>\Theta_0^{-1}\), then the map \(b \mapsto \lambda _ 1(b)\) is
    strongly diamagnetic, independently of the geometry of \(\Domain\).
  \end{enumerate}
\end{theorem}

Note that in Theorem~\ref{thm:maincritical}, we allow for $a$ to be negative. In fact, the sign of $a$ is irrelevant in the proof as long as $a>-1$.

\subsection{The supercritical regime}\label{subsec:supercritical}

We now consider a strong magnetic field that grows faster in the outer region
than in the inner region. For simplicity, we take \(\OuterB(b) = b ^ {\varrho}\),
with \(\varrho > 1\). Unlike the subcritical and critical regimes, we prove
strong diamagnetism independently of the geometry of the domain.

\begin{theorem}\label{thm:supercritical}
  If \(\OuterB(b) = b ^ {\varrho}\), where \(\varrho > 1\), then the map \(b
  \mapsto \lambda _ 1(b)\) is strongly diamagnetic.
\end{theorem}

\section{The perfect oscillatory regime}\label{sec:oscillatory}

\subsection{Introduction}

In this section, we prove Theorem~\ref{thm:mainoscillatory}. Along the way, we
also establish several auxiliary results of independent interest, including
extensions to higher eigenvalues and concentration properties of the
eigenfunctions.

The main steps of the proof consist in showing that, for large \(b\), the first
eigenvalue \(\lambda_1(b)\) is well approximated by \(\lambda_1^{\Outer}(b)\). We
then invoke a result from~\cite{HHOO} (stated below) describing the oscillatory
behavior of the latter.

Throughout this section, we work under the setting of
Theorem~\ref{thm:mainoscillatory}, namely we assume that the magnetic field is
given by
\begin{equation}\label{eq: mag field thm3.1}
  B(x)
  =
  \begin{cases}
    b      & x \in \Inner,\\
    a      & x \in \Outer,
  \end{cases}
\end{equation}
with \(a \in \mathbb{R}\) fixed.

\subsection{Periodicity and oscillations of eigenvalues for the outer domain}

The oscillations of \(\lambda _ 1 ^ {\Outer}(b)\) arise from a
periodicity property of \(\Hamiltonian ^ {\Outer}(b)\). Indeed, the parameter
\(b\)  determines the magnetic flux through the inner domain \(\Inner\), and integer values of the flux can be removed by a unitary gauge transformation. Since the flux through the
interior domain is given by \(b\measure{\Inner}/2\pi\), it follows that the eigenvalues \(\lambda _ n  ^ {\Outer}(b)\) are periodic in $b$ with period \(2\pi/\measure {\Inner}\). This is
a well-known result.

\begin{lemma}[{\cite[Theorem 1.1]{HHOO}}]\label{lem:HHOO}
  The operators \(\Hamiltonian ^ {\Outer}(b)\) and \(\Hamiltonian ^ {\Outer}(b -
  2\pi/\measure{\Inner})\) are unitarily equivalent. In particular, all eigenvalues
  \(\lambda _ n ^ {\Outer}(b)\) are \(2\pi/\measure{\Inner}\)-periodic in \(b\).

  Furthermore, all eigenvalues are continuous with respect to \(b\). Restricting \(b\) to \([0,2\pi/\measure{\Inner}]\), the first eigenvalue \(\lambda _
  1 ^ {\Outer}(b)\) attains its minimum at the end-points of the interval, and
  its (strict) maximum at the midpoint.
\end{lemma}

Since we will use the unitary mapping below, we elaborate a bit on it. Given any
point \(x _ * \in \Outer\), we define \(\Unitary \colon L ^ 2(\Outer) \to L ^
2(\Outer)\) as the multiplication operator
\begin{equation}\label{eq:unitary}
  (\Unitary\psi)(x)
  =
  \exp
    \left(
      \frac {2\pi \ii}{\measure{\Inner}}
      \int _ {\ell(x _ *,x)} \Fb _ {\Inner} \cdot \dd r
    \right)
    \psi(x).
\end{equation}
Here \(\ell(x _ *,x)\) is any regular path in \(\Outer\) connecting \(x _ *\) and
\(x\).

\begin{remark}[Path-independence of $\mathcal{U}$]
  The definition of $\mathcal{U}$ is independent of the chosen path. Indeed,
  given two curves, one can concatenate the first with the reverse of the second
  to form a loop \(\gamma \subset \Outer\).

  Since \(\Outer\) contains the hole \(\Inner\), loops in \(\Outer\) are
  classified by their winding number around \(\Inner\). If the winding number is
  \(n\), then by Stokes' theorem (recalling that \(\curl \Fb_{\Inner} =
  \mathbf{1}_{\Inner}\)) we have
  \[
    \int_\gamma \Fb_{\Inner} \cdot \dd r
    = n \int_{\Inner} \curl \Fb_{\Inner} \dd x
    = n\measure{\Inner}.
  \]
  It follows that
  \[
    \exp\Biggl(\frac{2\pi \ii}{|\Inner|} \int_\gamma \Fb_{\Inner} \cdot \dd r \Biggr) = 1,
  \]
  and hence \(\mathcal{U}\) is well-defined.
\end{remark}

A simple calculation shows that
\[
  \exp\left(-\frac {2\pi \ii}{\measure{\Inner}}\int _ {\ell(x _ *,x)} \Fb _ {\Inner} \cdot \dd r\right)
  \nabla\exp\left(\frac {2\pi \ii}{\measure{\Inner}}\int _ {\ell(x _ *,x)} \Fb _ {\Inner} \cdot \dd r\right)
  =
  \frac{2\pi\ii}{\measure{\Inner}}\Fb _ {\Inner},
\]
and hence it follows that
\[
  \Unitary ^ {-1} (-\ii\nabla - b\Fb _ {\Inner} - a\Fb _ {\Outer})\Unitary
  =
  -\ii\nabla - (b - 2\pi/\measure{\Inner})\Fb _ {\Inner} - a\Fb _ {\Outer}.
\]
Since \(\Fb_{\Inner}\) and \(\Fb_{\Outer}\) are both orthogonal to the normal
vector field on \(\partial\Domain\), the unitary mapping preserves the operator
domains.

More generally, for any integer $p$, we have
\begin{equation}\label{eq:unitarycalc}
  \Unitary ^ {-p} (-\ii\nabla - b\Fb _ {\Inner} - a\Fb _ {\Outer})\Unitary ^ {p}
  =
  -\ii\nabla - (b - 2\pi p/\measure{\Inner})\Fb _ {\Inner} - a\Fb _ {\Outer}.
\end{equation}
Choosing \(p = \integerpart{b\measure{\Inner}/2\pi}\) maps \(b\) into the
interval \([0,2\pi/\measure{\Inner})\), effectively reducing \(b\) modulo the
period of the eigenvalues. We henceforth call the interval
\([0,2\pi/\measure{\Inner})\) the \emph {base interval}.

\subsection{Asymptotic equivalence and proof of Theorem~\ref{thm:mainoscillatory}}

Our next result says that as \(b \to +\infty\) each eigenvalue of
\(\Hamiltonian(b)\) becomes arbitrarily close to the corresponding eigenvalue of
\(\Hamiltonian ^ {\Outer}(b)\).

\begin{proposition}\label{prop:asyequi}
  Let \(a \in \reals\) and \(n \in \naturalnumbers\) be fixed. Let the magnetic
  field be as in \eqref{eq: mag field thm3.1}. As \(b \to +\infty\), it holds
  that
  \begin{equation}
    \lambda _ n (b) = \lambda _ n ^ {\Outer} (b) + o(1).
  \end{equation}
\end{proposition}

The proof of Theorem~\ref{thm:mainoscillatory} is now immediate.

\begin{proof}[Proof of Theorem~\ref{thm:mainoscillatory}]
  The periodicity from Lemma~\ref{lem:HHOO} and the asymptotic equivalence in
  Proposition~\ref{prop:asyequi}, applied for \(n = 1\), directly imply that
  \(\lambda _ 1 (b)\) is not monotone on any unbounded interval of \(\reals _
  +\).
\end{proof}

We now prove Proposition~\ref{prop:asyequi}. By extending functions in the form
domain of \(\Hamiltonian^{\Outer}(b)\) by zero in \(\Inner\), it follows
immediately from the comparison of quadratic forms that \(\lambda_n(b) \leq
\lambda_n^{\Outer}(b)\).

Moreover, since the eigenvalues of \(\Hamiltonian^{\Outer}(b)\) are continuous
and periodic in \(b\), they are uniformly bounded with respect to \(b\). We
therefore conclude that, for each \(n\), there exists a constant \(\Lambda_n\)
such that, for all \(b\),
\begin{equation}\label{eq:simplebounds}
  \lambda_n(b) \leq \lambda_n^{\Outer}(b) \leq \Lambda_n.
\end{equation}
We need to work a bit more in order to obtain an inequality in the opposite
direction. We need a few properties of the eigenfunctions.

\begin{notation}\label{notation}
  For each \(b\) and each \(n \in \naturalnumbers\), we denote by \(u ^ {(n)}\)
  the \(n\)-th normalized eigenfunction of \(\Hamiltonian(b)\), and we set \(w ^
  {(n)}(x) = \Unitary ^ {-p} u ^ {(n)}(x)\), for \(x \in \Outer\) and \(p =
  \integerpart {b\measure{\Inner}/2\pi}\). Also, let \(\varepsilon _ 0 =
  \dist(\partial\Inner,\partial\Domain)\). For \(0 < \varepsilon < \varepsilon _
  {0}\) we introduce the sets \(\Outer _ {\varepsilon} = \set {x \in
  \Outer~:~\dist(x,\partial\Inner) > \varepsilon}\).
\end{notation}

In the next lemma, we prove that the eigenfunctions $u^{(n)}$ are essentially
localized in the outer domain \(\Outer\).

\begin{lemma}\label{lem:concentration}
  Let \(n \in \naturalnumbers\). Let \(u ^ {(n)}\) and \(w ^ {(n)}\) be as in
  Notation \ref{notation}. There exists \(b _ 0>0\) such that if \(b > b _ 0\),
  then
  \[
    \norm {\abs {u^{(n)}}} _ {H ^ 1(\Domain;\reals)} \leq C_n,
    \quad
    \text {and}
    \quad
    \norm {u^{(n)}} _ {L ^ 2(\Inner)} ^ 2 \leq C_n/b,
    \quad
    \text {and}
    \quad
    \norm {w^{(n)}} _ {L ^ 2(\Outer)} ^ 2 \geq 1 - C_n/b,
  \]
  where $C_n>0$ depends on $n$.
\end{lemma}

\begin{proof}
  Fix $n\in\mathbb{N}$. For simplicity, we use the notation $u= u^{(n)}$ and $w =
  w^{(n)}$. By the diamagnetic inequality and~\eqref{eq:simplebounds},
  \[
    \norm {\abs {u}} _ {H ^ 1(\Domain;\reals)} ^ 2
    \leq
    \norm {(-\ii\nabla - \Ab)u} _ {L ^ 2(\Domain)} ^ 2
    =
    \lambda _ n (b)
    \leq
    \Lambda _ n,
  \]
  so \(C_n = \Lambda _ n\) works for the first inequality. For the second
  inequality, let \(\lambda _ n ^ {\Inner}(b)\) denote the \(n\)th eigenvalue
  of \(\Hamiltonian ^ {\Inner}(b) = (-\ii\nabla - \Ab) ^ 2\) in \(L ^
  2(\Inner)\), with magnetic Neumann boundary conditions on \(\partial\Inner\).
  Then, by the triangle inequality and the Cauchy--Schwarz inequality (for all
  \(v\) in the form domain),
  \[
      \norm {(-\ii\nabla - \Ab)v} _ {L ^ 2(\Inner)} ^ 2
      \geq
      \frac{1}{2}\norm {(-\ii\nabla - b\Fb _ {\Inner})v} _ {L ^ 2(\Inner)} ^ 2
      -
      \norm {a \Fb _ {\Outer}} _ {L ^ {\infty}(\Inner)} ^ 2
      \norm {v} _ {L ^ 2(\Inner)} ^ 2.
  \]
  The first term, modulo the constant \(1/2\), is the quadratic form
  corresponding to the magnetic operator with a constant magnetic field in
  \(\Inner\). According to~\cite[Theorem~8.11]{FH-b}, as \(b \to +\infty\) its
  lowest eigenvalue, divided by \(b\), tends to \(\Theta _ 0/2 > 1/4\). Since
  the second term is bounded, we get for large \(b\) that \(\lambda _ 1 ^
  {\Inner}(b)/b\) is bounded below by some constant. Since \(\lambda _ n ^
  {\Inner}(b) \geq \lambda _ 1 ^ {\Inner}(b)\),
  \[
      \lambda _ 1 ^ {\Inner}(b)  \norm {u} _ {L ^ 2(\Inner)} ^ 2
      \leq
      \norm {(-\ii\nabla - \Ab)u} _ {L ^ 2(\Inner)} ^ 2
      \leq
      \norm {(-\ii\nabla - \Ab)u} _ {L ^ 2(\Domain)} ^ 2
      =
      \lambda _ n(b)
      \leq \Lambda _ n.
  \]
  Thus, there exist \(b _ 0\) and \(C > 0\) such that \(\norm {u} _ {L ^
  2(\Inner)} ^ 2 \leq C\Lambda_n/b\) if \(b > b _ 0\). The same constant works for the
  third inequality since it is only a different formulation of the second.
  Indeed, \(u\) is normalized and \(w\) differs only from \(u\) in \(\Outer\)
  by a multiplicative factor of absolute value one.
\end{proof}

We are now ready to prove Proposition~\ref{prop:asyequi}.

\begin{proof}[Proof of Proposition~\ref{prop:asyequi}]
  We divide the proof into several steps. The main idea is to reduce  \(b\)
  modulo the period, extract limits of the gauge-transformed eigenfunctions on
  \(\Outer\), show that the limits satisfy the Dirichlet condition on
  \(\partial\Inner\), and then apply the min-max principle on the limiting
  \(n\)-dimensional space.

  Fix \(n \in \naturalnumbers\). From~\eqref{eq:simplebounds} we have
  \(\lambda_n(b)\leq \lambda_n^{\Outer}(b)\) for all \(b\). Hence, it suffices to
  show that
  \begin{equation}\label{eq:goal-liminf}
    \liminf _ {b\to +\infty}
    \bigl(
      \lambda_n(b)-\lambda ^ {\Outer} _ n(b)\bigr) \geq 0.
  \end{equation}
  Choose a sequence \(b_k \to +\infty\) such that
  \[
    \lim_{k \to +\infty}\bigl(\lambda _ n(b _ k) - \lambda _ n ^ {\Outer} (b_k) \bigr)
    =
    \liminf_{b \to +\infty} \bigl(\lambda _ n(b) - \lambda ^ {\Outer} _ n(b)\bigr).
  \]
  For each \(k\), let $p_k = \integerpart {b_k\measure{\Inner}/2\pi}$; then
  \[
    \beta_k := b_k - \frac{2\pi p_k}{\measure{\Inner}}
    \in [0,2\pi/\measure{\Inner}).
  \]
  After passing to a subsequence, if necessary, by the
  Bolzano--Weierstrass theorem, we may assume that \(\beta_k \to \beta_*\) for some \(\beta_ *
  \in[0,2\pi/\measure{\Inner})\).
  For each \(k\), the operators \(\Hamiltonian ^ {\Outer}(b _ k)\) and
  \(\Hamiltonian ^ {\Outer}(\beta_k)\) are unitarily equivalent (via \(\Unitary ^
  {p _ k}\)). In particular,
  \begin{equation}\label{eq:asyequi-outer-period}
    \lambda_j^{\Outer}(b_k)=\lambda_j^{\Outer}(\beta_k)
    \quad\text{for all }j\in\naturalnumbers.
  \end{equation}
  We denote by \(\Ab _ k\) and \(\Ab _ *\) the magnetic vector potentials that appear
  in \(\Hamiltonian(\beta _ k)\) and \(\Hamiltonian(\beta _ *)\), respectively,
  \begin{equation}\label{eq:asyequi-potentials}
    \Ab_k \coloneq \beta_k\Fb_{\Inner}+a\Fb_{\Outer},
    \qquad
    \Ab_* \coloneq \beta_*\Fb_{\Inner}+a\Fb_{\Outer}.
  \end{equation}
  For each \(k\), we further let \(u ^ {(j)} _ k\) be eigenfunctions
  of \(\Hamiltonian(b_k)\) corresponding to \(\lambda_j(b_k)\), \(1\leq j\leq
  n\), chosen to be orthonormal in \(L^2(\Domain)\). Following Notation~\ref{notation}, we write
  \[
    w ^ {(j)} _ k := \Unitary^{-p_k} \bigl(u ^ {(j)} _ k \big|_{\Outer}\bigr).
  \]
  
  \begin{lemma}\label{lem:asyequi-H1}
    For each fixed \(j\in\{1,\dots,n\}\), the sequence \((w ^ {(j)} _ k)\) is
    uniformly bounded in \(H ^ 1(\Outer)\).
  \end{lemma}

  \begin{proof}
    First we note that \(\norm {w ^ {(j)} _ k}_{L^2(\Outer)} = \norm {u ^ {(j)} _
    k} _ {L^2(\Outer)} \leq 1\). Next, since \(\beta_k\) belongs to the base
    interval, the sequence \((\beta_k)_{k\in\mathbb{N}}\) is bounded. Since
    \(\Fb_{\Inner}\) and \(\Fb_{\Outer}\) are continuous on \(\closure{\Outer}\),
    by the definition of $\Ab_k$ in \eqref{eq:asyequi-potentials}, there exists
    \(C>0\) such that \(\norm {\Ab _ k} _ {L ^ \infty(\Outer)} \leq C\) for all
    \(k\). Moreover, by unitary invariance and \eqref{eq:simplebounds},
    \[
      \begin{multlined}
      \int_{\Outer}\abs{(-\ii\nabla-\Ab_k)w_k^{(j)}}^2\dd x
      =
      \int_{\Outer}\abs{(-\ii\nabla-\Ab)u_k^{(j)}}^2\dd x
      \\
      \leq
      \int_{\Domain}\abs{(-\ii\nabla-\Ab)u_k^{(j)}}^2\dd x
      =
      \lambda_j(b_k)
      \leq \Lambda_j.
      \end{multlined}
    \]
    By the triangle inequality and the Cauchy--Schwarz inequality we can turn
    this into a uniform bound for the gradients,
    \[
      \int_{\Outer}\abs{\nabla w ^ {(j)} _ k} ^ 2 \dd x
      \leq
      2\Lambda_j + 2C^2.
    \]
    Thus for each \(j \in \set{1,\ldots,n}\) we conclude that \(\norm {w_k^{(j)}} _
    {H ^ 1(\Outer)}\) is uniformly bounded in \(k\).
  \end{proof}

  \begin{lemma}\label{lem:asyequi-L2}
    For each \(j\in\{1,\dots,n\}\) there exist \(w_*^{(j)}\in H^1(\Outer)\) and a
    subsequence (not relabeled) such that
    \[
      w ^ {(j)} _ k \weakto  w ^ {(j)} _ *
      \quad \text{in }H^1(\Outer),
      \qquad
      w ^ {(j)} _ k \to w ^ {(j)} _ * \quad \text{in }L^2(\Outer).
    \]
    Moreover, \(w ^ {(j)} _ *\) belongs to the form domain of
    \(\Hamiltonian^{\Outer}(\beta_*)\), that is \(w ^ {(j)} _ * \in H^1(\Outer)\) and
    \(w ^ {(j)} _ *|_{\partial\Inner} = 0\) (in the trace sense).
  \end{lemma}

  \begin{proof}
    Fix \(j\in\{1,\dots,n\}\). By Lemma~\ref{lem:asyequi-H1}, \(\set{w_k^{(j)}}\)
    is contained in a closed ball in \(H ^ 1(\Outer)\). The Banach--Alaoglu
    theorem guarantees existence of a weakly convergent subsequence.
    Therefore, there exist some \(w ^ {(j)} _ * \in H ^ 1(\Outer)\) such that
    \[
      w ^ {(j)} _ k \weakto  w ^ {(j)} _ * \quad \text {in } H^1(\Outer).
    \]

    Since \(\Outer\) is bounded with regular boundary, the embedding
    \(H^1(\Outer)\hookrightarrow L^2(\Outer)\) is compact (Rellich--Kondrachov).
    Thus, after extracting a further subsequence if needed, we have
    \(w ^ {(j)} _ k \to w ^ {(j)} _ *\) strongly in \(L ^ 2(\Outer)\).

    We prove next that the trace of \(w ^ {(j)} _ *\) vanishes on \(\partial\Inner\).
    According to Lemma~\ref{lem:concentration} there exists \(C _ j > 0\) such that
    \begin{equation}\label{eq:inner-mass}
      \norm {u ^ {(j)} _ k} _ {L ^ 2(\Inner)} ^ 2
      \leq
      \frac {C _ j}{b_k}.
    \end{equation}
    Note that the constant $C_j>0$ is uniform in $k$ since it is uniform with
    respect to the intensity of the magnetic field (see
    Lemma~\ref{lem:concentration}). In particular, following the proof of
    Lemma~\ref{lem:concentration}, we can also deduce that
    \begin{equation}\label{eq: est L2norm uj from 1 to n}
        \|u^{(j)}_k\|^2_{L^2(\Inner)} \leq \frac{C_n}{b_k}, \qquad \forall j=1,\ldots,n.
    \end{equation}
    Let \(\gamma \colon H^1(\Inner)\to L^2(\partial\Inner)\) denote the trace
    map. By the multiplicative trace inequality (see, e.g.,~\cite[Theorem
    1.6.6]{BrSc}) there exists \(C_{\mathrm {tr}} > 0\) such that for all \(v \in
    H^1(\Inner)\),
    \begin{equation}\label{eq:trace-L2-interp}
      \norm {\gamma v} _ {L^2(\partial\Inner)}^2
      \leq
      C_{\mathrm {tr}}
      \norm {v} _ {L ^ 2(\Inner)}
      \norm {v} _ {H ^ 1(\Inner)}.
    \end{equation}
    We apply~\eqref{eq:trace-L2-interp} with \(v = \abs {u ^ {(j)} _ k\vert_V}\),
    and get, by Lemma~\ref{lem:concentration}, that
    \begin{equation}\label{eq:trace-modulus-zero}
      \gamma \abs{u ^ {(j)} _ k} \to 0
      \quad
      \text{in }L ^ 2(\partial\Inner).
    \end{equation}
    Since \(u ^ {(j)} _ k \in H ^ 1(\Domain)\), the traces of \(\abs {u ^ {(j)} _
    k}\) from \(\Inner\) and \(\Outer\) coincide on \(\partial\Inner\). Since
    \(\abs {w ^ {(j)} _ k} = \abs {u ^ {(j)} _ k}\) in \(\Outer\), and since one
    can interchange the trace and absolute value,
    \[
      \norm {\gamma ^ {\Outer} w ^ {(j)} _ k} _ {L ^ 2(\partial\Inner)}
      \to
      0,
      \quad\text {as } k \to +\infty,
    \]
    where we denoted by $\gamma^{\Outer}$ the trace map defined for $\partial \Outer$.
    On the other hand, since \(w_k^{(j)}\weakto w_*^{(j)}\) weakly in
    \(H^1(\Outer)\) and the trace map \(H^1(\Outer)\to L^2(\partial\Inner)\) is
    continuous, we have
    \[
      \gamma^{\Outer} w_k^{(j)} \weakto \gamma^{\Outer} w_*^{(j)}
      \quad\text{weakly in }L^2(\partial\Inner).
    \]
    By uniqueness of the weak limit, it follows that
    \[
      \gamma^{\Outer} w_*^{(j)}=0.
    \]
    Hence \(w_*^{(j)}\) satisfies the Dirichlet condition on \(\partial\Inner\),
    and therefore \(w_*^{(j)}\) belongs to the form domain of
    \(\Hamiltonian^{\Outer}(\beta_*)\).
  \end{proof}

  By a diagonal extraction, we may assume that the conclusions of Lemma~\ref{lem:asyequi-L2}
  hold simultaneously for all \(j \in \set {1,\ldots,n}\).

  \begin{lemma}\label{lem:asyequi-ON}
    The family \(\{w ^ {(j)} _ *\}_{j = 1} ^ n\) is orthonormal in
    \(L^2(\Outer)\), i.e.,
    \[
      \innerproduct{ w_*^{(j)},w_*^{(m)}} _ {L ^ 2(\Outer)}
      =
      \delta_{jm}.
    \]
  \end{lemma}
  \begin{proof}
    For \(j\neq m\),
    \[
      \innerproduct{ w_k^{(j)},w_k^{(m)} } _ {L ^ 2(\Outer)}
      =
      \innerproduct{ \Unitary^{-p_k}u_k^{(j)},\Unitary^{-p_k}u_k^{(m)} } _ {L ^ 2(\Outer)}
      =
      -
      \innerproduct{ u_k^{(j)},u_k^{(m)} } _ {L ^ 2(\Inner)}.
    \]
    by using the orthogonality of $(u_k^{(j)})$ in $L^2(\Domain)$.
    The Cauchy--Schwarz inequality and Lemma~\ref{lem:concentration} imply that, as \(k
    \to +\infty\),
    \[
      \abs {
        \innerproduct{
          w_k^{(j)},
          w_k^{(m)}
        } _ {L ^ 2(\Outer)}
      }
      \to
      0.
    \]
    On the other hand, by Lemma~\ref{lem:concentration}, we also have that
    \[
      1
      \geq
      \norm {w ^ {(j)} _ k} _ {L^2(\Outer)}^2
      \geq 1-\frac{C_j}{b_k},
    \]
    and so \(\norm {w ^ {(j)} _ k} _ {L^2(\Outer)}^2 \to 1\) as \(k \to
    +\infty\). By the continuity of the inner product and norm, and by the strong
    \(L^2(\Outer)\)-convergence in Lemma~\ref{lem:asyequi-L2} we conclude the
    claimed orthonormality.
  \end{proof}

  We recall that the quadratic form \(\Quadraticform ^ {\Outer} _ {b}\) corresponding to the
  operator \(\Hamiltonian ^ {\Outer}(b)\) is given by
  \[
    \Quadraticform ^ {\Outer} _ {b}[\psi]
    \coloneq
    \int_{\Outer}\abs{(-\ii\nabla-b\Fb_{\Inner}-a\Fb_{\Outer})\psi}^2\dd x,
    \quad
    \psi\in H^1(\Outer),\ \psi|_{\partial\Inner}=0.
  \]
  We will make use of the space
  \[
      S _ * ^n\coloneq
      \Span \set{w ^ {(1)} _ *, \ldots, w ^ {(n)} _ * } \subset L^2(\Outer).
  \]
  By Lemma~\ref{lem:asyequi-ON} the dimension of \(S _ * ^n\) is \(n\).

  \begin{lemma}\label{lem:asyequi-liminf-minmax}
    It holds that
    \[
      \lambda_n^{\Outer}(\beta_*) \leq \liminf_{k\to\infty}\lambda_n(b_k).
    \]
  \end{lemma}

  \begin{proof}
    By the min--max principle for the Neumann--Dirichlet operator \(\Hamiltonian^{\Outer}(\beta_*)\),
    \begin{equation}\label{eq:minmax-restrict}
      \lambda_n^{\Outer}(\beta_*)
      \leq
      \max_{0\neq \psi\in S _ * ^n}
      \frac
        {\Quadraticform ^ {\Outer} _ {\beta_*}[\psi]}
        {\norm {\psi } _ {L ^ 2(\Outer)} ^ 2}.
    \end{equation}
    It therefore suffices to prove that for every \(\psi\in S _ * ^n\setminus\set{0}\),
    \begin{equation}\label{eq:goal-rayleigh}
      \frac
        {\Quadraticform ^ {\Outer} _ {\beta_*}[\psi]}
        {\norm {\psi} _ {L ^ 2(\Outer)} ^ 2}
      \leq
      \liminf_{k\to\infty}\lambda_n(b_k),
    \end{equation}
    and then take the maximum over \(\psi\).

    Fix \(\psi_n\in S _ * ^n\). Write
    \[
      \psi_n = \sum_{j = 1} ^ n c _ j w ^ {(j)} _ *
    \]
    for some coefficients \(c_1,\ldots,c_n\in\complexes\), not all zero, and define
    \[
      \psi_{k,n}
      \coloneq
      \sum_{j = 1} ^ n c _ j w ^ {(j)} _ k.
    \]
    By Lemma~\ref{lem:asyequi-L2},
    \begin{equation}\label{eq:psi-conv}
      \psi_{k,n}\to\psi_n \ \text{in }L^2(\Outer),
      \quad
      \psi_{k,n}\weakto \psi_n \ \text{in }H^1(\Outer).
    \end{equation}
    In particular,
    \begin{equation}\label{eq:psi-norm}
      \|\psi_{k,n}\|_{L^2(\Outer)}\to \|\psi_n\|_{L^2(\Outer)}.
    \end{equation}
    We claim that
    \begin{equation}\label{eq:psi-liminf-energy}
      \Quadraticform ^ {\Outer} _ {\beta_ *}[\psi_n]
      \leq
      \liminf_{k\to\infty}
        \Quadraticform ^ {\Outer} _ {\beta_ k}[\psi_{k,n}].
    \end{equation}
    Since \(\Ab_k \to \Ab_*\) uniformly on \(\closure {\Outer}\) we find, by also
    using~\eqref{eq:psi-norm}, that
    \begin{equation}\label{eq:A-diff}
      \|(\Ab_k-\Ab_*)\psi_{k,n}\|_{L^2(\Outer)}
      \to 0,
      \quad
      k \to +\infty.
    \end{equation}
    From~\eqref{eq:psi-conv} we also get that
    \begin{equation}\label{eq:weak-conv-magnetic}
      (-\ii\nabla-\Ab_*)\psi_{k,n} \weakto  (-\ii\nabla-\Ab_*)\psi_n
      \quad\text{in }L^2(\Outer).
    \end{equation}
    By weak lower semicontinuity of the \(L^2\)-norm,
    \begin{equation}\label{eq:wls-psi}
      \Quadraticform ^ {\Outer} _ {\beta_ *}[\psi_n]
      =
      \norm {(-\ii\nabla-\Ab_*)\psi_n } _ {L ^ 2(\Outer)} ^ 2
      \leq
      \liminf_{k\to\infty}\|(-\ii\nabla-\Ab_*)\psi_{k,n}\|_{L^2(\Outer)}^2.
    \end{equation}
    We also need to replace \(\Ab _ *\) by \(\Ab _ k\) in the right-hand side,
    and we use the uniform convergence \(\Ab _ k \to \Ab _ *\) again, to verify
    that it is possible. Indeed, by the triangle inequality
    \[
      \norm {(-\ii\nabla-\Ab_*)\psi_{k,n}} _ {L ^ 2(\Outer)}
      \leq
      \norm {(-\ii\nabla-\Ab_k)\psi_{k,n} } _ {L ^ 2(\Outer)}
      +
      \norm {(\Ab_k-\Ab_*)\psi_{k,n}} _ {L ^ 2(\Outer)}.
    \]
    Hence, for any \(\eta > 0\), by the weighted Cauchy--Schwarz inequality,
    \[
      \begin{aligned}
        \norm {(-\ii\nabla-\Ab_*)\psi_{k,n}} _ {L ^ 2(\Outer)} ^ 2
        &\leq
        (1 + \eta)\norm {(-\ii\nabla-\Ab_k)\psi_{k,n}} _ {L ^ 2(\Outer)} ^ 2
        \\
        &\qquad +
        (1 + 1/\eta)
        \norm {(\Ab_k-\Ab_*)\psi_{k,n}} _ {L ^ 2(\Outer)} ^ 2.
      \end{aligned}
    \]
    By first passing to \(\liminf _ {k\to+\infty}\) and then letting \(\eta \to 0^+\),
    we find that
    \[
      \liminf _ {k \to +\infty} \norm {(-\ii\nabla-\Ab_*)\psi_{k,n}} _ {L ^ 2(\Outer)} ^ 2
      \leq
      \liminf _ {k \to +\infty} \norm {(-\ii\nabla-\Ab_k)\psi_{k,n}} _ {L ^ 2(\Outer)} ^ 2.
    \]
    Combining this estimate with~\eqref{eq:wls-psi} we
    get~\eqref{eq:psi-liminf-energy}.

    To proceed, we define the function \( {u} _ {k,n}\) as the combination of
    the eigenfunctions \(u ^ {(j)} _ k\) of \(\Hamiltonian(b _ k)\), with same
    coefficients \(c _ j\) as above,
    \[
     {u} _ {k,n}
      \coloneq
      \sum_{j = 1} ^ n c _ j u ^ {(j)} _ k
      \in
      \Span\{u ^ {(1)} _ k,\ldots,u ^ {(n)} _ k\}
      \subset L^2(\Domain).
    \]
    Since the \(u ^ {(j)} _ k\) are orthonormal eigenfunctions of \(\Hamiltonian(b_k)\),
    \[
      \int_{\Domain}\abs{(-\ii\nabla-\Ab) {u} _ {k,n}}^2 \dd x
      =
      \sum_{j=1}^n \abs {c _ j} ^ 2 \lambda _ j(b _ k)
      \leq
      \lambda _ n(b _ k) \sum_{j = 1} ^ n \abs {c _ j} ^ 2
      =
      \lambda _ n(b _ k) \norm {{u} _ {k,n} } _ {L ^ 2(\Domain)} ^ 2.
    \]
    Applying \(\Unitary^{-p_k}\) does not change norms or magnetic energies, and
    restricting the integral to \(\Outer\) can only decrease it. Hence
    \begin{equation}\label{eq:outer-energy-bound}
       \Quadraticform ^ {\Outer} _ {\beta _ k}[{\psi} _ {k,n} ]
      =
      \int_{\Outer}\abs{(-\ii\nabla-\Ab_k){\psi} _ {k,n} }^2 \dd x
      \leq
      \lambda _ n(b _ k) \norm {{u} _ {k,n} } _ {L ^ 2(\Domain)} ^ 2.
    \end{equation}

    We need also to compare \(\norm {{u} _ {k,n} } _ {L ^ 2(\Domain)}\) and
    \(\norm {{\psi} _ {k,n} } _ {L ^ 2(\Outer)}\). Since the unitary operator
    \(\Unitary^{-p_k}\) is multiplication with a unimodular function on
    \(\Outer\), we have \(\norm {{u} _ {k,n} } _ {L ^ 2(\Outer)} = \norm {{\psi} _ {k,n} } _ {L ^ 2(\Outer)}\), and therefore
    \[
      \norm {{u} _ {k,n} } _ {L ^ 2(\Domain)} ^ 2
      =
      \norm {{\psi} _ {k,n} } _ {L ^ 2(\Outer)} ^ 2
      +
      \norm {{u} _ {k,n} } _ {L ^ 2(\Inner)} ^ 2.
    \]
    By \eqref{eq: est L2norm uj from 1 to n}, we know that there exists a constant $C_n>0$ such that
    \(\|u_k^{(j)}\|_{L^2(\Inner)}^2\le C_n/b_k\) for each \(j\in\set{1,\ldots,n}\).
    Hence, by the Cauchy--Schwarz inequality,
    \[
      \norm {{u} _ {k,n} }_{L ^ 2(\Inner)}
      \leq \sum_{j=1}^n \abs {c _ j} \norm {u ^ {(j)} _ k} _ {L^2(\Inner)}
      \leq \Bigl(\sum_{j = 1} ^ n \abs {c _ j} ^ 2\Bigr) ^ {1/2}
           \Bigl(\sum_{j = 1} ^ n \norm {u ^ {(j)} _ k} _ {L ^ 2(\Inner)} ^ 2\Bigr) ^ {1/2}
      \leq \norm{c}_{\ell^2} \Bigl(\frac{C_n n}{b _ k}\Bigr) ^ {1/2}.
    \]
    Thus \(\norm {{u} _ {k,n}} _ {L ^ 2(\Inner)} ^ 2 \leq (C_nn/b _ k)\norm
    {c} _ {\ell ^ 2} ^ 2\).

    From the proof of Lemma~\ref{lem:asyequi-ON} we find that the Gram
    matrices \((\innerproduct { w ^ {(j)} _ k,w ^ {(m)} _ k } _ {L ^ 2(\Outer)})
    _ {j,m=1}^n\) converge to the identity matrix as \(k \to +\infty\). In
    particular, for \(k\) large enough, the Gram matrix $G_k = (\langle
    w_k^{(j)}, w^{(m)}_k\rangle_{L^2(\Outer)})_{j,m=1}^n$ is uniformly positive
    definite and such that $G_k\rightarrow \mathrm{Id}$ as $k\rightarrow\infty$.
    Thus, there exists \(c _ 0 > 0\) independent of \(k\) such that
    \begin{equation}\label{eq:psi-lower}
      \norm {{\psi} _ {k,n}} _ {L ^ 2(\Outer)} ^ 2
      =
      \sum _ {j , m=1} ^ n
        c _ j\conjugate{c _ m}
        \innerproduct { w ^ {(j)} _ k,w ^ {(m)} _ k } _ {L ^ 2(\Outer)}
      \geq
      c _ 0 \sum _ {j = 1} ^ n \abs {c _ j} ^ 2
      =
      c _ 0 \norm {c} _ {\ell ^ 2} ^ 2.
    \end{equation}
    Consequently, as \(k \to +\infty\), for fixed $n$, we have
    \[
      \frac{\norm {{u} _ {k,n} } _ {L ^ 2(\Domain)} ^ 2}
           {\norm {{\psi} _ {k,n}} _ {L ^ 2(\Outer)}^2}
      =
      1 + \frac{\norm{{u} _ {k,n}} _ {L ^ 2(\Inner)} ^ 2}
               {\norm {{u} _ {k,n}} _ {L ^ 2(\Outer)}^2}
      \leq
      1 + \frac{(C_n n/b_k)\norm {c} _ {\ell ^ 2} ^ 2}
               {c _ 0 \norm {c} _ {\ell ^ 2} ^ 2}
      =
      1 + O(1/b_k)
      =
      1 + o(1).
    \]
    Inserting this into~\eqref{eq:outer-energy-bound} yields
    \begin{equation}\label{eq:Qk-Rayleigh}
      \frac{\Quadraticform ^ {\Outer} _ {\beta _ k}[{\psi} _ {k,n}]}
           {\norm {{\psi} _ {k,n}} _ {L ^ 2(\Outer)} ^ 2}
      \leq
      \lambda_n (b_k) \frac{\norm {{u} _ {k,n}} _ {L ^ 2(\Domain)} ^ 2}
                           {\norm {{\psi} _ {k,n}} _ {L ^ 2(\Outer)} ^ 2}
      \leq
      \lambda_n(b_k)\bigl(1 + o(1)\bigr).
    \end{equation}

    From~\eqref{eq:psi-liminf-energy},~\eqref{eq:psi-norm},
    and~\eqref{eq:Qk-Rayleigh}, we obtain
    \[
      \frac{\Quadraticform ^ {\Outer} _ {\beta_ *}[{\psi} _ {n}]}
           {\norm {{\psi} _ {n}} _ {L ^ 2(\Outer)} ^ 2}
      \leq
      \liminf_{k\to\infty}
      \frac{\Quadraticform ^ {\Outer} _ {\beta _ k}[{\psi} _ {k,n}]}
           {\norm {{\psi} _ {k,n}} _ {L ^ 2(\Outer)} ^ 2}
      \leq
      \liminf_{k\to\infty}
      \lambda _ n(b _ k).
    \]
    This is exactly~\eqref{eq:goal-rayleigh}.
  \end{proof}

  We now have everything we need to finish the proof. By the unitary equivalence
  of \(\Hamiltonian ^ {\Outer} (b _ k)\) and \(\Hamiltonian ^ {\Outer} (\beta
  _k)\) we have \(\lambda_n^{\Outer}(b_k) = \lambda_n^{\Outer}(\beta_k)\). Since
  \(b\mapsto\lambda_n^{\Outer}(b)\) is continuous and \(\beta_k\to \beta_*\), it
  follows that
  \[
    \lambda_n^{\Outer}(b_k)=\lambda_n^{\Outer}(\beta_k)\to \lambda_n^{\Outer}(\beta_*).
  \]
  Together with Lemma~\ref{lem:asyequi-liminf-minmax}, we infer
  \[
    \liminf_{k\to\infty}\bigl(\lambda_n(b_k)-\lambda_n^{\Outer}(b_k)\bigr)
    =
    \liminf_{k\to\infty}\lambda_n(b_k)-\lim_{k\to\infty}\lambda_n^{\Outer}(b_k)
    \geq 0,
  \]
  which implies the inequality in~\eqref{eq:goal-liminf}. Since $\psi_n\in
  S^{n}_\ast\setminus \{0\}$ is arbitrary, we conclude the proof of
  Proposition~\ref{prop:asyequi}.
\end{proof}

\subsection{Two consequences}\label{subsec:profile}

As a consequence of the proof of Proposition~\ref{prop:asyequi}, using PDE
techniques, we can upgrade the convergence of eigenfunctions modulo phase shifts.
For a fixed $\beta_\ast \in [0,2\pi/\measure{\Inner})$, let $\{p_k\}_{k\in \naturalnumbers}$ be a sequence such
that $p_k\in \integers$ and $p_k\rightarrow \infty$ and let
\begin{equation}\label{eq:def-bk}
b_k:=\beta_*+\frac{2\pi}{\measure{\Inner}}p_k \quad\text{diverges as $k\to+\infty$,}
\end{equation}
and for all $k\in\naturalnumbers$, let $\{u_k^{(n)}\}_{n\in\naturalnumbers}$ be
the orthonormal set of eigenfunctions of $\Hamiltonian (b_k)$, and
$\{w_k^{(n)}\}_{n\in\naturalnumbers}$ be the corresponding functions defined via
the phase shift. More precisely, following Notation~\ref{notation}, we write
$w_k^{(j)} \coloneq \mathcal{U}^{-p_k}(u^{(j)}_k\vert_{\Outer})$.

\begin{proposition}[Convergence of eigenfunctions]\label{prop:conv-ef}
  Let  $0\leq \beta_*<2\pi/\measure{\Inner}$ be given and consider a sequence $\{b_k\}_{k\in\naturalnumbers}$ as in \eqref{eq:def-bk}. For a given positive integer $n$, there exists an orthonormal set $\{w^{(j)}_*\}_{1\leq j\leq n}$ of eigenfunctions of $\Hamiltonian ^ {\Outer}(\beta_*)$ and a subsequence (not relabeled) such that, for all $j\in\{1,2,\ldots,n\}$, we have
  \begin{equation}\label{eq: conv mod H1}
      \abs {u^{(j)}_k}\to \abs {w_*^{(j)}}\text{ in }H^1(\Outer;\reals),
  \end{equation}
  \begin{equation}\label{eq: conv min in H1omega0}
      w_k^{(j)}\to w_*^{(j)}\text{ in }H^1(\Outer;\complexes),
  \end{equation}
  and
  \[
    \Hamiltonian ^{\Outer}(\beta_*)w_*^{(j)}
    =
    \lambda_j^{\Outer}(\beta_*)w_*^{(j)}
    \quad\text{in $\Outer$},
    \quad w_*^{(j)}|_{\partial\Inner}=0.
  \]
\end{proposition}

Another consequence of Proposition~\ref{prop:asyequi} concerns the behavior of
higher energy levels. In fact, Proposition~\ref{prop:asyequi} can be extended to
eigenvalues whose index grows with the flux (that is, $n\to\infty$ as
$b\to\infty$). More precisely, for a given $\lambda>0$, we define
\begin{equation}\label{eq: def Ilambdab}
  \mathcal I_\lambda(b)
  =
  \{n\in\naturalnumbers\colon\lambda_n^{\Outer}(b)<\lambda b\}.
\end{equation}
We will need the de Gennes constant $\Theta_0\in(\frac12,1)$, whose definition is
recalled in Appendix~\ref{sec:dG}. In particular, we will assume that
$\lambda<\Theta_0$. This condition is related to the transition between boundary
and bulk magnetic states. Indeed, $\Theta_0 b$ is the natural energy scale
associated with boundary localization for the Neumann magnetic Laplacian.
Eigenfunctions with energy below this threshold remain localized near the
boundary of $\Outer$, which prevents penetration into the strongly magnetic region
$\Inner$. As a consequence, the comparison with $\Hamiltonian^{\Outer}(b)$ remains valid
uniformly for $n\in\mathcal {I}_\lambda(b)$.

Using Proposition~\ref{prop:asyequi}, Dirichlet--Neumann bracketing, and the Weyl
law for the Laplacian, we obtain the following result.

\begin{proposition}[Asymptotics of high energy levels]\label{prop:ev-high-N}
  Let $a\in\reals$ and $0<\lambda<\Theta_0$ be fixed. Let
  $\mathcal{I}_\lambda(b)$ be as in \eqref{eq: def Ilambdab}. As $b\to+\infty$,
  the following holds
  \begin{equation}\label{eq: weyl n}
    \sup_{n\in\mathcal I_\lambda(b)}\abs {\frac{\lambda_n(b)}{\lambda_n^{\Outer}(b)}-1 }\to 0
    \text{ as }b\to+\infty.
  \end{equation}
\end{proposition}

For the convenience of the reader, we present the proofs of
Propositions~\ref{prop:conv-ef} and \ref{prop:ev-high-N} in
Appendix~\ref{app:proofs}.

\section{The subcritical regime}\label{sec:subcritical}

\subsection{Introduction}

We consider in this section \(\OuterB(b) =  b ^ {\varrho}\) for some \(0 < \varrho < 1\). This
means that the magnetic field \(B\) is now given by
\begin{equation}\label{eq: mag field subcrit}
  B(x)
  =
  \begin{cases}
    b          & x \in \Inner,\\
    b ^ {\varrho} & x \in \Outer.
  \end{cases}
\end{equation}
The corresponding flux \(\Phi\) through the domain \(\Domain\) is given by \(\Phi
= b\measure{\Inner}/2\pi + b ^ {\varrho}\measure{\Outer}/2\pi\). In particular the
fluxes through both the inner domain \(\Inner\) and the outer domain \(\Outer\)
will become infinite as \(b \to +\infty\), but the flux through the inner part is
still considerably stronger, and that will in fact be sufficient to enforce
oscillations for the disk but not for generic domains.

The main goal of this section is to give a proof of Theorem~\ref{thm:mainsubcritical}. To do
this we will first discuss the asymptotic behavior of the first eigenvalue
\(\lambda _ 1(b)\) of \(\Hamiltonian(b)\), as \(b \to +\infty\). Once that
is settled, we will be able to discuss the monotonicity question.

\subsection{Asymptotic behavior of the first eigenvalue}

We first give a one-term asymptotic expansion for the first eigenvalue, valid for
any domain \(\Domain\). We recall that \(\Theta _ 0 \approx 0.59\) is
the spectral constant from the de Gennes problem, see Appendix~\ref{sec:dG}.

We show that the leading term in the expansion for $\lambda_1(b)$ as
$b\rightarrow \infty$ is given by $\Theta_0 b^\varrho$. This behavior can be heuristically understood from the fact that low-energy states prefer to localize in the region where the magnetic field is weaker, namely the outer
region $\Outer$, where the field has strength $b^\varrho$. Since $0<\varrho<1$, the magnetic
field in $\Outer$ is much smaller than the field $b$ in $\Inner$, which acts as an energetic barrier.

\begin{proposition}\label{prop:subcriticaleig}
Let $0<\varrho<1$ and let the magnetic field be as in \eqref{eq: mag field subcrit}.  The lowest eigenvalue \(\lambda_1(b)\) of $\Hamiltonian(b)$ satisfies, as \(b\to+\infty\)
  \[
    \lambda_1(b) = \Theta_0 b^{\varrho} + \ordo(b ^ {\varrho}).
  \]
\end{proposition}

\begin{proof}
  We start with the lower bound. Let \(u \in H^1(\Domain)\) with $\|u\|_{L^2(\Domain)} = 1$. We write
  \[
    \Quadraticform[u] = \int_{\Domain} |(-\ii\nabla - \Ab)u|^2 = \int_{\Inner} |(-\ii\nabla - \Ab)u|^2 + \int_{\Outer} |(-\ii\nabla - \Ab)u|^2,
  \]
  and we estimate each of the two integrals above separately. By \cite[Proposition~8.2.2]{FH-b} there exists \(C > 0\)
  such that, for large \(b\),
  \[
    \norm {(-\ii\nabla-\Ab)u} _ {L ^ 2(\Outer)} ^ 2
    \geq
    (\Theta _ 0 b ^ {\varrho} - Cb ^ {3\varrho/4})\norm {u} _ {L ^ 2(\Outer)} ^ 2
  \]
  and
  \[
    \norm {(-\ii\nabla-\Ab)u} _ {L ^ 2(\Inner)} ^ 2
    \geq
    (\Theta _ 0 b - Cb ^ {3/4})\norm {u} _ {L ^ 2(\Inner)} ^ 2.
  \]
Using that $b$ is large and that $\norm {u} _ {L^2(\Domain)} ^ 2 = \norm {u} _
{L^2(\Inner)} ^ 2 + \norm {u} _ {L^2(\Outer)} ^ 2 = 1$, we can deduce that
\[
    \Quadraticform[u] \geq \Theta_0 b^\varrho - Cb^{3\varrho/4} + (\Theta_0(b - b^\varrho) - C(b^{3/4} - b^{3\varrho/4}))\|u\|^2_{L^2(\Inner)} \geq   \Theta_0 b^\varrho - Cb^{3\varrho/4},
\]
where we also used that $0<\varrho<1$. Thus
  \[
    \lambda _ 1(b) \geq \Theta _ 0 b ^ {\varrho} - \Ordo(b ^ {3\varrho/4}).
  \]
  We now prove the matching upper bound. In order to do that, we construct a trial state that is localized close to the outer boundary $\partial \Domain$, inside $\Outer$.

Let $t = \dist(x, \partial\Domain)$ and let \(s\) be the tangential coordinate along $\partial \Domain$.  We choose a normalized trial state \(u\) as
  in~\cite[Eq. (8.4)]{FH-b}, which has the following structure
  \begin{equation}\label{eq:trial-state}
    u
    =
    c
    \varphi_0(b^{\varrho/2}t)
    \times
    \ee^{\ii\xi_0b^{\varrho/2}s}
    \times
    \chi(s,t)
    \times
    \text{gauge transformation},
  \end{equation}
  where \(\varphi_0\) is the eigenfunction of the de Gennes model
  (Appendix~\ref{sec:dG}), $\xi_0$ is the minimizing parameter in the de Gennes model (Appendix~\ref{sec:dG}), and $\chi$ is a smooth cut-off chosen so that the support of \(u\) is within \(\Outer\). The gauge transformation that we do not write explicitly above is used to reduce the operator to the half-line model to which $\Theta_0$ is connected. The constant $c$ guarantees that $u$ is normalized.  With this trial state we find, for large \(b\),
  that
  \[
    \lambda_1(b)
    \leq
    \Quadraticform[u]
    =
    \Quadraticform ^ {\Outer} _ b[u]
    \leq
    \Theta _ 0 b ^ {\varrho} + \ordo(b ^ {\varrho}).
  \]
  Combining the lower and upper bounds completes the proof.

\end{proof}

As a consequence of the first-term asymptotics we get an Agmon-type estimate
for the ground state, saying that as \(b\) becomes large, the ground state is
localized at the outer boundary \(\partial\Domain\).

\begin{lemma}\label{lem:subagmon}
Let $0<\varrho<1$ and let the magnetic field be as in \eqref{eq: mag field subcrit}. There exist positive constants
  \(C_1,b_1\) and \(\delta_1 < 1\) such that every normalized ground state \(u\)
  of \(\lambda_1(b)\) satisfies for all \(b\geq b_1\),
  \[
    \int_{\Domain}
      \Bigl(b^{-\varrho} \abs {(-\ii\nabla-\Ab)u}^2 + \abs {u}^2 \Bigr)
      \ee^{\delta_1 b ^ {\varrho/2} \dist(x,\partial\Domain)}
      \dd x
    \leq
    C_1.
  \]
\end{lemma}

We omit the proof, since it follows exactly the same scheme as for the uniform
magnetic field, see~\cite[Section~8.2.3]{FH-b}.\medskip

\subsection{The subcritical regime for the unit disk}

In this section, we analyze the subcritical regime when \(\Domain\) is the unit disk. By translation and scaling, this is the only disk that needs to be considered. In particular, we provide a more accurate asymptotic expansion for \(\lambda_1(b)\), involving the constants \(\Theta_0\), \(\xi_0\), and \(\varphi_0(0)\) introduced in Appendix~\ref{sec:dG}. This result will allow us to prove the strong non-diamagnetism stated in Theorem~\ref{thm:mainsubcritical}.

\begin{proposition}\label{prop:disk-subcritical}
  Suppose that $\Domain\subset\mathbb{R}^2$ is the unit disk. Let $0<\varrho<1$ and let the magnetic field be as in \eqref{eq: mag field subcrit}.
  There exist constants \(C_0\) and \(C_1\) such that if
  \begin{equation}\label{eq: deltab}
    \Delta_b
    \coloneq
    \inf_{m \in \integers}
    \left|m - \Flux - \xi_0 b ^ {\varrho/2} - C_0 \right|
  \end{equation}
  then, as \(b \to +\infty\), the lowest eigenvalue \(\lambda _ 1(b)\) of
  \(\Hamiltonian(b)\) satisfies
  \[
    \lambda_1(b)
    =
    \Theta_0 b ^ {\varrho}
    - \frac13\varphi_{0}(0)^2 b^{\varrho/2}
    + \xi_0\varphi_{0}(0)^2 \bigl((\Delta_b)^2 + C_1 \bigr)
    + \Ordo(b^{-\varrho/2}).
  \]
\end{proposition}

\begin{proof}
  After a few reductions, the proof follows from Sections~4 and 5 in~\cite{FS}.

  We introduce for convenience \(\OuterB \coloneq b^{\varrho}\). By the Agmon estimate (Lemma~\ref{lem:subagmon}), the ground state $u$ satisfies
  \[
    \int_U e^{c\sqrt{\sigma}\dist(x,\partial U)}|u(x)|^2 \dd x \leq C.
  \]
  Hence, for any fixed $\eta>0$, the contribution of the region $\set{x \mid \dist(x,\partial \Domain)\geq \sigma^{-\frac{1}{2} + \eta}}$ to the energy is $\Ordo(\sigma^{-\infty})$.

  Therefore, $\lambda_1(b)$ coincides up to  $\Ordo(\sigma^{-\infty})$ with the lowest eigenvalue of the operator restricted to the annulus
  \[
  \mathcal{A}_\sigma := \{1- c\sigma^{-1/2 + \eta} < |x| <1\},
  \]
  with Dirichlet boundary conditions on the interior circle and Neumann boundary conditions on the outer circle, and with
  vector potential
  \[
    \Ab = b \Fb_{\Inner} + \sigma \Fb_{\Outer}.
  \]
  We now follow the idea that since up to errors of the order
  $\Ordo(\sigma^{-\infty})$, the energy is concentrated in the region $\Outer$,
  we can think of $\Ab$ as the vector potential associated with a
  constant magnetic field with intensity $\sigma$, i.e., $\Ab \sim (\sigma r/2)
  (-\sin \theta, \cos\theta)=: \sigma\Ab_0$. We can implement this
  rigorously by performing a gauge transformation. The error we incurred
  replacing $\Ab$ by $(\sigma r/2) (-\sin \theta, \cos\theta)$ in a
  rigorous way is an Aharonov-Bohm potential carrying most of the flux. Indeed,
  if we set $\widetilde{\Ab} = \Ab - \sigma\Ab_0$,
  then $\curl
  \widetilde{\Ab} = 0$ in $\mathcal{A}_\sigma$. It is always true that, in
  an annulus, any curl-free vector field can be written as
  \[
    \widetilde{\Ab} = \nabla\varphi + \alpha \Ab _ {\mathrm {AB}},
  \]
  where $\varphi$ is a scalar function, $\Ab _ {\mathrm {AB}} = r^{-1}(-\sin
  \theta, \cos\theta)$ 
  and $\alpha$ is a constant that measures the circulation.
  Since the circulation of $\Ab$ equals $2\pi \Phi$ and that of $\Ab _ 0$ is
  $\pi\sigma$, we set $\alpha = \Phi - \sigma/2$, which guarantees that the
  circulation of $\widetilde{\Ab}$ is $0$.

  After switching to polar coordinates and separating variables, we are left with
  a family of radial operators \(\Hamiltonian _ m(b)\), indexed by the angular
  momentum \(m\in \integers\). The operators \(\Hamiltonian _ m(b)\) act in the
  weighted space \(L^2((0,1),r\dd r)\) as
  \[
  \Hamiltonian _ m(b)
  = - \frac{\dd^2}{\dd r^2}
    - \frac{1}{r}\frac{\dd}{\dd r}
    + \frac{1}{r^2}(m - \sigma r a(r))^2,
  \]
  where we have
  \[
  ra(r) = \Flux/\sigma + (r - 1) + \Ordo\bigl((r-1)^2\bigr).
  \]
  This is the same family of operators studied in \cite[Sections~4 and~5]{FS}.
  In particular, it has been shown that the smallest eigenvalue of
  \(\Hamiltonian _ m(b)\) is given by
  \[
    \Theta_0 b^{\varrho}
    - \frac13\varphi_{0}(0)^2 b^{\varrho/2}
    + \xi_0\varphi_{0}(0)^2\bigl((\Delta_b)^2 + C_1 \bigr)
    + \Ordo (b^{-\varrho/2})
  \]
  if \(\mu_2 \coloneqq m - (\Flux + \xi_0\sigma^{1/2})\) satisfies
  \(\mu_2 = \Ordo(\sigma^{1/4})\).

  On the other hand, for some positive constants \(A\) and \(C <
  \varphi_{0}(0)^2/3\) and \(m\) satisfying \(|\mu_2| > A \sigma^{1/4}\), the
  smallest eigenvalue is greater than \(\Theta_0\sigma - C \sigma^{1/2}\) for
  \(\sigma\) sufficiently large.
\end{proof}

\subsection{The  subcritical regime for a generic domain}\label{subsec:subcritical-monotonicity}

Recalling that the domain of $\Hamiltonian(b)$ is independent of $b$ (see
\eqref{eq: def domain Hb}), using analytic perturbation theory (see
\cite[Section~2.1]{FH}), we deduce that for all \(b\) the following one-sided
derivative exists,
\[
  \lambda_{1,+}'(b)
  =
  \lim_{\varepsilon \to 0^+}
  \frac{\lambda_1(b + \varepsilon) - \lambda_1(b)}
       {\varepsilon}.
\]
It is convenient to introduce $\sigma = \sigma(b) =
b^{\varrho}$ as a parameter and look at the eigenvalue dependence on $\sigma$.
More precisely, we set
\begin{equation}\label{eq: def lambdasigma lambda'sigma}
    \lambda(\sigma):= \lambda_1(\sigma^{1/\varrho}), \qquad \lambda^\prime_+(\sigma) :=  \lim_{\varepsilon \to 0^+}
  \frac{\lambda(\sigma + \varepsilon) - \lambda(\sigma)}
       {\varepsilon}.
\end{equation}
Note that the existence of $\lambda^\prime_+(\sigma)$ can be proved as the
existence of $\lambda^\prime_{1,+}(b)$.

In the next result, we study the behavior of
$\lambda_{+}'(\sigma)$ as $\sigma\rightarrow \infty$, from which we can deduce
the monotonicity of $\lambda_1(b)$ for large $b$ as stated in part~(2) of
Theorem~\ref{thm:mainsubcritical}.

\begin{proposition}\label{prop:sub-generic-monotone}
  Suppose that $\Domain\subset \mathbb{R}^2$ is not a disk. Let $0<\varrho<1$ and let the magnetic field be as in \eqref{eq: mag field subcrit}. Let $\sigma = b^{\varrho}$ and let $ \lambda(\sigma)$, $\lambda^\prime_+(\sigma)$ be as in \eqref{eq: def lambdasigma lambda'sigma}.
  Then
  \[
    \liminf_{\sigma \to +\infty}\lambda_{+}'(\sigma) \geq \Theta_0 > 0.
  \]
\end{proposition}

\begin{proof}[Proof of Proposition~\ref{prop:sub-generic-monotone}]
  The proof relies on the Agmon estimate from
  Lemma~\ref{lem:subagmon} and follows the strategy of the proof
  of~\cite[Theorem~8.5.1]{FH-b}. We start by introducing a suitable gauge
  transformation. We split the proof in several steps.

  \noindent\textbf{Step 1 (Change of gauge).} Let $\kappa_{\max}$ be the maximum
  of the curvature $\kappa$ on $\partial\Domain$ and introduce the set
  $n(\partial\Domain):=\set{x\in\partial\Domain\colon \kappa(x) =
  \kappa_{\max}}$. Since $\Domain$ is not a disk, $n(\partial\Domain)\not =
  \partial\Domain$.

  We now define a global gauge transformation to express the
  vector potentials $\Fb_{\Outer}$ and $\Fb_{\Inner}$ appearing in
  \eqref{eq:def-F} in a convenient form. We first work in the region $\Outer$. Let
  $\widehat{\Outer} \subset \Outer$ be a simply connected domain such that
  \begin{equation}\label{eq:sub-generic-domaininclusion}
    \closure{\Domain} \setminus \widehat{\Outer}
    \subset
    \set{x \in \closure {\Domain}~:~\dist(x,n(\partial\Domain)) > \varepsilon _ 0},
  \end{equation}
  where the constant $\varepsilon_0$ is positive and depends on the domain
  $\Domain$. Then there exists a smooth function $\phi_W \in C^\infty(\Domain)$ such
  that $\widehat{\Fb}_{\Outer} = \Fb _ {\Outer} + \nabla \phi_{\Outer}$ satisfies
  \begin{equation}\label{eq:sub-generic-hatF-outer}
    \abs {\widehat{\Fb} _ {\Outer}}
    \leq C \dist(x,\partial\Domain)
    \text{ on }\widehat{\Outer}.
  \end{equation}
  This construction of $\phi_W$ is detailed in \cite[Lemma~8.5.5]{FH-b} and a
  closer look at the proof shows that $\phi_{\Outer}$ is independent of $\sigma$
  and that $\widehat{\Fb} _ {\Outer} = \Fb _ {\Outer}$ on $\Inner$ since $\phi_{\Outer}
  = 0$ in $\Inner$.

  In the region $\Inner$ instead, we define
  \[
    \widetilde{\phi}_{\Inner}(x) = - \int_{\gamma(x_0, x)} \Fb_{\Inner}\dd r \qquad x\in \widehat{\Outer},
  \]
  where $x_0\in \widehat{\Outer}$ is fixed and $\gamma(x_0,x)\subset \widehat{\Outer}$ is a
  smooth path joining $x_0$ and $x$. Since $\curl \Fb_{\Inner} = 0$ on
  $\widehat{\Outer}$, we deduce that $\widetilde{\phi}_{\Inner}$ is independent of the path, and also that
  \[
      \nabla\widetilde{\phi}_{\Inner}
      =
      - \Fb_{\Inner} \qquad \text{in } \widehat{\Outer}.
  \]
  Since we want to define a global gauge transformation, we extend now
  $\widetilde{\phi}_{\Inner}$ in a smooth way to $\Domain$ by using $\chi\in
  C^\infty(\Domain)$ such that $\chi \equiv 1$ in $\widehat{\Outer}$ and
  $\supp\chi \subset \Outer$. Thus, we set $\phi_{\Inner} =
  \widetilde{\phi}_{\Inner} \chi$. This choice guarantees that

  the vector potential $\widehat{\Fb}_\Inner = \Fb_{\Inner} +
  \nabla\phi_{\Inner}$ satisfies
  \begin{equation}\label{eq:sub-generic-hatF-inner}
    \widehat{\Fb}_{\Inner} = 0
    \quad\text{on } \widehat{\Outer}.
  \end{equation}
  Thus, we introduce the vector potential

  \begin{equation}\label{eq:Asigma}
    {\Ab}_\sigma
    :=
    \sigma^{1/\varrho}\widehat{\Fb} _ {\Inner}
    +
    \sigma \widehat{\Fb} _ {\Outer} ,
  \end{equation}
  which by construction is gauge equivalent to $\Ab$. For the rest of the proof
  we let $\psi$ be a normalized ground state of $(-\ii\nabla - {\Ab}_\sigma)^2$.
  The corresponding ground state energy is $\lambda(\sigma)$.

  \noindent\textbf{Step 2 (Asymptotics for $\lambda(\sigma)$ and localization of
  the ground state)}. We now prove that as \(\sigma \to
  +\infty\),

   \begin{equation}\label{eq:lambdasigmaexpansion}
      \lambda(\sigma) =
      \Theta_0\sigma
      -
      \frac13\varphi_0(0)^2\kappa_{\mathrm{max}}\sigma^{1/2}
      +
      \ordo(\sigma^{1/2}).
   \end{equation}

  Note that the first two terms in \eqref{eq:lambdasigmaexpansion} are consistent
  with the case of the disk. However, unlike
  Proposition~\ref{prop:disk-subcritical}, there is no flux effect coming from
  the interior because the corresponding ground states are localized near the
  boundary points of maximal curvature. This set is a loop when the domain is a
  disk. Under the assumptions of Proposition~\ref{prop:sub-generic-monotone}, the
  ground states concentrate near $\partial\Domain$, thanks to
  Lemma~\ref{lem:subagmon}, and we have

   \[
      \lambda(\sigma)
      =
      \lambda^{\Outer}(\sigma)
      +
      \Ordo(\sigma^{-\infty}),
    \]

  where $\lambda^{\Outer}(\sigma)$ is the lowest eigenvalue of $\Hamiltonian ^
  {\Outer}(\sigma)= (-\ii\nabla - \sigma \widehat{\Fb}_{\Outer})^2$ in
  $L^2(\Outer)$, with Dirichlet boundary condition on $\partial\Inner$. Since
  $\curl(\sigma \widehat{\Fb}_{\Outer}) = \sigma$ in $\Outer$ (see
  \eqref{eq:prop-A} and \eqref{eq:Asigma}), we get by~\cite[Theorem~8.3.2]{FH-b}
  \[
    \lambda^{\Outer}(\sigma)
    =
    \Theta_0\sigma
    -
    \frac13\varphi_0(0)^2 \kappa _ {\mathrm {max}}\sigma ^ {1/2}
    +
    \ordo(\sigma^{1/2}).
  \]
  Notice that the Dirichlet boundary condition on the interior boundary of
  $\Outer$ does not alter the conclusion in \cite[Theorem~8.3.2]{FH-b}, because
  the localization is close to $\partial\Domain$. We conclude
  that~\eqref{eq:lambdasigmaexpansion} is true.

  Knowing that \eqref{eq:lambdasigmaexpansion} holds, we can refine the
  localization of the ground state. In fact, by \cite[Proposition~8.3.3]{FH-b}
  we see that the ground states of $\lambda(\sigma)$ are localized near the set
  $n(\partial\Domain)$. This localization can roughly be described as follows
  \cite[Lemma~8.5.3 and Lemma~8.5.4]{FH-b}
  \begin{equation}\label{eq: localiz subcrit}
    \int_{\Domain}
      [\dist(x,\partial\Domain)]^N \abs{\psi}^2\dd x
    =
    \Ordo(\sigma^{-N/2}),
    \quad
    \int_{\set{\dist(x,n(\partial\Domain)) > \varepsilon_0}}
      \abs{\psi}^2\dd x
    =
    \Ordo(\sigma^{-N}),
  \end{equation}
  for fixed $N \in \naturalnumbers$, $\varepsilon_0>0$, and for $\psi$ a
  normalized ground state of $\lambda(\sigma)$.

  \noindent\textbf{Step 3 (Estimate for $\lambda^\prime_+(\sigma)$).}
  We now prove that for \(\varepsilon > 0\) it holds that
  \begin{equation}\label{eq: est lambda'+}
    \lambda_{+}'(\sigma)
    \geq
    \frac{\lambda(\sigma + \varepsilon) - \lambda(\sigma)}{\varepsilon}
    -
    \varepsilon\norm {\Rb\psi} ^ 2 + E,
  \end{equation}
  where
  \begin{equation}\label{eq:sub-generic-perturbation-R}
    \Rb
    =
    \frac {1}
          {\varepsilon}
          (\Ab _ {\sigma + \varepsilon} - \Ab _ {\sigma}),
  \end{equation}
  and
  \begin{equation}\label{eq:sub-generic-perturbation-E}
    E
    =
    2\re\innerproduct{(\Rb - \partial _ {\sigma}\Ab _ {\sigma})\psi,
                      (-\ii\nabla - \Ab _ {\sigma})\psi}.
  \end{equation}

  In order to prove \eqref{eq: est lambda'+}, we combine analytic perturbation
  theory with a comparison of the quadratic forms associated with
  $\lambda(\sigma)$ and $\lambda(\sigma + \varepsilon)$. We first select an
  analytic eigenbranch realizing the ground state energy near $\sigma$, and then
  apply the Feynman--Hellmann formula along this branch.

  Let \(\sigma>0\) be fixed. Since the form domain is independent of the
  parameter and the coefficients of the operator depend real-analytically on
  \(\sigma\), the family of magnetic Neumann Laplacians is an analytic family (in
  the sense of Kato) with compact resolvent. Hence, in a neighborhood of
  \(\sigma\), the eigenvalues can be parameterized by finitely many real-analytic
  branches \(\mu_j(\tau)\), with corresponding normalized real-analytic
  eigenfunctions \(\phi_j(\tau)\). Among the branches satisfying
  \(\mu_j(\sigma)=\lambda(\sigma)\), choose one and denote it by \(\mu(\tau)\),
  with corresponding normalized analytic eigenfunction denoted by \(\phi(\tau)\),
  in such a way that the Taylor expansion of \(\mu\) at \(\tau=\sigma\) is
  lexicographically minimal. By analyticity, after possibly shrinking the
  neighborhood, this branch coincides with the lowest eigenvalue for all
  \(\tau\ge \sigma\) close to \(\sigma\); that is, there exists \(\delta>0\) such
  that
  \[
    \mu(\tau)=\lambda(\tau), \qquad \tau\in[\sigma,\sigma+\delta).
  \]
  Set \(\psi=\phi(\sigma)\). Then \(\psi\) is a normalized ground state of
  \(\lambda(\sigma)\), and
  \[
    \lambda_+'(\sigma)=\mu'(\sigma).
  \]
  Applying the Feynman--Hellmann formula to the branch
  \((\mu(\tau),\phi(\tau))\), we have
  \[
    \lambda_{+}'(\sigma)
    =
    -\innerproduct{
      \partial_\sigma \Ab _ \sigma \psi,
      (-\ii\nabla - \Ab _ \sigma)\psi
      }
    -\innerproduct{(-\ii\nabla - \Ab _ \sigma)\psi,
      \partial_\sigma\Ab _ \sigma \psi
      }.
  \]
  Using the min-max principle for $\lambda(\sigma + \varepsilon)$, we have
  \[
    \begin{aligned}
      \MoveEqLeft[1]
      \lambda(\sigma+\varepsilon)-\lambda(\sigma)
      \\
      &
      \leq
      \innerproduct{ (-\ii\nabla - \Ab _ {\sigma + \varepsilon})\psi,
                     (-\ii\nabla - \Ab _ {\sigma + \varepsilon})\psi
                    }
      -
      \innerproduct{ (-\ii\nabla - \Ab_\sigma)\psi,
                     (-\ii\nabla - \Ab_\sigma)\psi
                    }
      \\
      &
      =
      \varepsilon^2 \norm {\Rb \psi} ^ 2
      -
      \varepsilon
      \left(
        \innerproduct { \Rb \psi,
                        (-\ii\nabla - \Ab_\sigma)\psi
                      }
        +
        \innerproduct {(-\ii\nabla - \Ab _ \sigma)\psi ,
                       \Rb \psi
                      }
      \right)
      \\
      &
      =
      \varepsilon^2 \norm {\Rb \psi} ^ 2
      +
      \varepsilon \lambda _ {+} '(\sigma)
      -
      \varepsilon E.
    \end{aligned}
  \]
  Dividing by $\varepsilon>0$ and rearranging, we get the claimed inequality.

  \noindent\textbf{Step 4 (Estimates for the gauge-transformed vector potential).}
  As \(\sigma \to +\infty\),
  \begin{equation}\label{eq:sub-generic-gauges}
    \norm {\widehat{\Fb}_ {\Outer}\psi}^2
    =
    \Ordo(\sigma^{-1})
    \quad\text{and}\quad
    \norm {\widehat{\Fb}_ {\Inner}\psi }^2
    =
    \Ordo(\sigma^{-1 - 2/\varrho}).
  \end{equation}

  In light of \eqref{eq:sub-generic-domaininclusion} and
  \eqref{eq:sub-generic-hatF-outer}, and the foregoing
  localization \eqref{eq: localiz subcrit},
  \begin{align*}
    \norm {\widehat{\Fb}_{\Outer}\psi} ^ 2
    &=
    \int_{\widehat{\Outer}}\abs {\widehat{\Fb} _ {\Outer}} ^ 2 \abs{\psi} ^ 2\dd x
    +
    \int_{\Domain\setminus\widehat{\Outer}}
      \abs {\widehat{\Fb} _ {\Outer}} ^ 2
      \abs {\psi} ^ 2 \dd x
  \\
  &\leq C\int_{\widehat{\Outer}} |\dist(x,\partial \Domain)|^2|\psi|^2 + \|\widehat{\Fb}_{\Outer}\|_\infty^2\int_{\{\dist(x, n(\partial \Domain))> \varepsilon_0\}} |\psi|^2 \dd x \leq C\sigma^{-1}.
  \end{align*}
  Finally, in light of~\eqref{eq:sub-generic-hatF-inner}, we have
  \[
    \norm {\widehat{\Fb} _ {\Inner} \psi} ^ 2
    =
    \int_{\Domain\setminus\widehat{\Outer}}
      |{\widehat{\Fb} _ {\Inner}}| ^ 2 \abs {\psi} ^ 2 \dd x.
  \]
  Using~\eqref{eq:sub-generic-domaininclusion} and the aforementioned decay of
  $\psi$ away from $n(\partial\Domain)$ (with $N \geq 1 + 2/\varrho$), we get
  \[
    \norm {\widehat{\Fb} _ {\Inner} \psi} ^ 2
   \leq \|\widehat{\Fb}_{\Inner}\|_\infty^2\int_{\{\dist(x, n(\partial \Domain))> \varepsilon_0\}} |\psi|^2 \dd x
    \leq C\sigma^{-1 - 2/\varrho}.
  \]
  This establishes \eqref{eq:sub-generic-gauges}.

  \noindent\textbf{Step 5 (Conclusions).} We now want to give
  a lower bound for $\lambda^\prime_+(\sigma)$. To start, we use that, by the
  expansion of \(\lambda(\sigma)\) in \eqref{eq:lambdasigmaexpansion}, we find
  that there exists a constant $C_0>0$ such that
  \[
    \frac{\lambda(\sigma + \varepsilon) - \lambda(\sigma)}{\varepsilon}
    \geq
    \Theta_0 - C_0\sigma^{-\frac{1}{2}}.
  \]
  We now bound the term $\varepsilon \|\mathbf{R}\psi\|^2$. By the definition of
  \(\Rb\) in \eqref{eq:sub-generic-perturbation-R}, it follows that
  \[
    \Rb
    =
    \biggl(
      \frac{(\sigma+\varepsilon)^{1/\varrho} - \sigma^{1/\varrho}}
           {\varepsilon}
    \biggr)
    \widehat {\Fb} _ {\Inner}
    +
    \widehat {\Fb} _ {\Outer}.
  \]
  Let \(\varepsilon = \eta\sigma\), where \(\eta \in (0,1)\). Then, by the mean
  value theorem,
  \begin{equation}\label{eq: mean value subcr}
    \frac{(\sigma + \varepsilon) ^ {1/\varrho} - \sigma ^ {1/\varrho}}{\varepsilon}
    \leq
    \frac{1}{\varrho} 2 ^ {1/\varrho - 1}\sigma ^ {1/\varrho - 1}.
  \end{equation}
  
  Consequently, there exists a constant \(C _ \varrho\) depending on $\varrho$, such
  that
  \begin{equation}\label{eq:sub-generic-normestimates}
    \norm {\Rb \psi}
    \leq
    C_\varrho\sigma^{1/\varrho - 1}
    \norm {\widehat{\Fb} _ {\Inner} \psi}
    +
    \norm {\widehat{\Fb}_ {\Outer}\psi},
  \end{equation}
  which implies, using \eqref{eq:sub-generic-gauges}, that for fixed $\varrho$,
  there exists a constant $C_1>0$ such that
  \[
    \varepsilon \|\mathbf{R}\psi\|^2 \leq C_1\sigma\eta (\sigma^{\frac{2}{\varrho} - 2} \|\Fb_{\Inner}\psi\|^2 + \|\Fb_{\Outer}\psi\|^2) \leq C_1\sigma^{-2} + C_1\eta.
  \]
  We now estimate $E$. To start, we use \(
    \partial _ {\sigma} \Ab _ {\sigma}
    =
    (\sigma ^ {1/\varrho-1}/\varrho)\widehat{\Fb} _ {\Inner}
    +
    \widehat{\Fb} _ {\Outer}
  \)
  to write \(E\) as
  \[
    E
    =
    2
    \biggl(
      \frac{(\sigma+\varepsilon)^{1/\varrho} - \sigma^{1/\varrho}}
           {\varepsilon}
      -
      \frac{\sigma^{1/\varrho - 1}}
           {\varrho}
    \biggr)
    \re\innerproduct { \widehat {\Fb}_{\Inner} \psi,
                       (-\ii\nabla - \Ab _ \sigma)\psi}.
  \]
  Using again \eqref{eq: mean value subcr}, we see that there exists a constant
  $\widetilde{C}_\varrho>0$ depending on $\varrho$ and an overall constant
  $C_2>0$ such that
  \[
  \abs {E}
    \leq
    \widetilde{C}_\varrho\sigma^{1/\varrho - 1}
    \norm {\widehat {\Fb} _ {\Inner}\psi}
    \sqrt{\lambda(\sigma)} \leq C_2\sigma^{\frac{1}{\varrho}- \frac{1}{2}} \norm {\Fb_{\Inner}\psi} \leq C_2\sigma^{-1},
  \]
  where we used \eqref{eq:sub-generic-gauges} and that $|\lambda(\sigma)| =
  \Ordo(\sigma)$ from \eqref{eq:lambdasigmaexpansion}. Combining all of
  these bounds and inserting them into \eqref{eq: est lambda'+}, we find that for each
  fixed \(\eta \in (0,1)\) there exists a constant $C>0$ independent of $\eta$
  and a threshold \(\sigma _ {\eta}\) such that for \(\sigma \geq \sigma _
  {\eta}\),
  \begin{equation}\label{eq:lb-der-mu}
      \lambda _ {+}'(\sigma)
      \geq
      \Theta _ 0
      -
      C \eta
      -
      C\sigma ^ {-2}
      -
      C\sigma ^ {-1}
  \end{equation}
  By first sending \(\sigma \to +\infty\), and then \(\eta \to 0 ^ {+}\), we get
  the claimed lower bound for \(\lambda _ {+}'(\sigma)\).
\end{proof}

\subsection{Proof of Theorem \ref{thm:mainsubcritical}}
In this section, we prove Theorem~\ref{thm:mainsubcritical}. From the asymptotic expansion in Proposition~\ref{prop:disk-subcritical}, we observe that the first two terms, $\sim b^\varrho+b^{\varrho/2}$,
are monotone functions of \(b\), whereas the term $\xi_0\varphi_{0}(0)^2\bigl((\Delta_b)^2+C_1\bigr)$ is oscillatory. Since \(0<\varrho<1\), the derivatives of the monotone terms vanish as \(b\to+\infty\). Consequently, their variation cannot compensate for the oscillations generated by \((\Delta_b)^2\). This leads to the non-monotonicity statement in Theorem~\ref{thm:mainsubcritical} when \(\Domain\) is a disk.
For a general domain, i.e., not a disk, the situation is different. Heuristically, the result is explained by the fact that the curvature has a strict maximum on \(\partial\Domain\). In this regime, the ground states localize near the points of maximal curvature, and this localization mechanism leads to the eventual monotonicity of \(\lambda_1(b)\); see~\cite[Section~8.5]{FH-b}.

\begin{proof}[Proof of Theorem \ref{thm:mainsubcritical}]

\textbf{Proof of (1) (Disk case).} In order to prove the non-monotonicity in the disk case, we prove that
  there exists an increasing sequence \((b _ k)\) such that \(b _ k \to +\infty\)
  and
  \[
    \lambda_1(b_k)
    <
    \lambda_1(b_{k + 1}),
    \quad
    \lambda_1(b_{k + 1})
    >
    \lambda_1(b_{k + 2}).
  \]
  From this, we can indeed conclude that \(b \mapsto \lambda _ 1(b)\) is not monotone in any unbounded
  interval of the form \((b _ 0,+\infty)\).
To construct such a sequence, we start by noticing that from Proposition \ref{prop:disk-subcritical}, we know that
\[
\lambda_1(b)
=
M(b)
+
\xi_0\varphi_{0}(0)^2(\Delta_b)^2
+
O(b^{-\varrho/2}),
\]
where we set
\[
M(b)
:=
\Theta_0 b^\varrho
-\frac13\varphi_{0}(0)^2 b^{\varrho/2}
+\xi_0\varphi_{0}(0)^2 C_1
\]
and $\Delta_b$ is defined as in \eqref{eq: deltab}.
Since $\Phi(b)
=
\frac{\measure{\Inner}}{2\pi}b+\frac{\measure{\Outer}}{2\pi}b^\varrho$,
the function
\[
f(b):=\Phi(b)+\xi_0 b^{\varrho/2}+C_0
\]
is continuous, strictly increasing, and satisfies $ f(b)\to+\infty$ as
$b\to+\infty$. Hence there exists an increasing sequence $(b_k)$, with \(b _ {k +
1} - b _ k\) bounded, such that
\[
  f(b_{2k})\in\integers,\qquad f(b_{2k+1})\in\integers + \frac12.
\]
Therefore
\[
  \Delta_{b_{2k}}=0,\qquad\Delta_{b_{2k+1}}=\frac12.
\]
Moreover,
\[
  M'(b)=\Theta_0\varrho b^{\varrho-1}-\frac{\varrho}{6}\varphi_{0}(0)^2 b^{\varrho/2-1}\to0
\]
as $b\to+\infty$, thus the variation of the non-oscillating terms between
consecutive points $b_k,b_{k+1},b_{k+2}$ becomes arbitrarily small for large $k$.
On the other hand, $(\Delta_{b_{2k+1}})^2-(\Delta_{b_{2k}})^2=1/4$. Hence, for
sufficiently large $k$,
\[
  \lambda_1(b_{2k})<\lambda_1(b_{2k+1}), \quad \text {and}\quad
  \lambda_1(b_{2k+1})>\lambda_1(b_{2k+2}).
\]

\noindent\textbf{Proof of (2) (Non-disk domain).}
 In this case, we can use  Proposition~\ref{prop:sub-generic-monotone} to
obtain that \(\sigma\mapsto \lambda(\sigma)\) is monotone increasing on
\([\sigma_0,+\infty)\) for \(\sigma_0 > 0\) sufficiently large, since \(\sigma
\mapsto \lambda(\sigma)\) is locally Lipschitz. In turn, we obtain that
\(b\mapsto \lambda_1(b)\) is monotone increasing on \([b_0,+\infty)\) with \(b_0
= \sigma_0^{1/\varrho}\), thereby proving the monotonicity statement of generic
\(\Domain\) in Theorem~\ref{thm:mainsubcritical}.
\end{proof}

\begin{remark}[The case of a sign-changing field]\label{rem:subcrit-sign-chan}
We remark that our methods allow us to deal with the case $\sigma(b)=-b^{\varrho}$, that is
\[
  B(x)=
  \begin{cases}
     b           & x\in \Inner,\\
    -b^{\varrho} & x\in \Outer.
  \end{cases}
\]
The flux becomes $\Phi=b\measure {\Inner}/2\pi - b^{\varrho}\measure {\Outer}/2\pi$
and is still dominated by the contribution from the inner domain $\Inner$. Modulo
slight adjustments of the proof, Proposition~\ref{prop:disk-subcritical} (with
the modified flux $\Phi$) and Proposition~\ref{prop:sub-generic-monotone} hold.
Thus,
\begin{itemize}
  \item For the disk, the lowest eigenvalue is not monotone and undergoes indefinite oscillations \cite[Section~6, Proof of Theorem~1.7]{FS}.
  \item For a generic domain, the lowest eigenvalue is monotone in a neighborhood of $+\infty$.
\end{itemize}
\end{remark}

\section{The critical regime}\label{sec:critical}

\subsection{Introduction}

We let \(a \neq 0\) and consider here \(\OuterB(b) = a b\), i.e.\ magnetic fields of the form
\begin{equation}\label{eq: mag field crit}
  B(x)
  =
  \begin{cases}
    b  & x \in \Inner,\\
    ab & x \in \Outer.
  \end{cases}
\end{equation}
Our ultimate aim is to prove Theorem~\ref{thm:maincritical}.
In contrast with the subcritical regime, the magnetic field now has the same
order of magnitude in both regions \(\Inner\) and \(\Outer\). As a consequence,
the leading order of the ground state energy is determined by a competition
between three different localization mechanisms:
\begin{itemize}
\item localization near the interface \(\partial\Inner\), governed by the
magnetic step model and producing the constant \(\hat\beta_a\),
\item localization near the outer boundary \(\partial\Domain\), producing the
effective energy \(|a|\Theta_0\), where \(\Theta_0\) is the de~Gennes constant already used in the subcritical regime (see Appendix~\ref{sec:dG}),
\item localization in the bulk of \(\Inner\), producing the energy of the magnetic Laplacian in the plane with constant magnetic field of intensity $b$.
\end{itemize}
Here, \(\hat\beta_a\) denotes the bottom of the spectrum of the magnetic step
operator introduced in Appendix~\ref{sec:ms}. The leading asymptotics of
\(\lambda_1(b)\) is therefore determined by the smallest among these three
effective energies. Depending on the value of \(a\), different localization
regimes occur.

\subsection{Asymptotic behavior of the first eigenvalue}

Just as in the subcritical regime, we start by giving a first-term asymptotic of
the lowest eigenvalue \(\lambda _ 1(b)\). This result will again imply a
localization result for the ground state.

\begin{proposition}\label{prop:criticaleig}
  Suppose that $a \neq 0$ and that \(\OuterB(b) = ab\).

  If $a < -1$, then, as \(b \to +\infty\),
  \[
    \lambda_1(b) = \hat{\beta} _ a b + \ordo(b).
  \]

  If $-1 \leq a < \Theta _ 0 ^ {-1}$ then, as \(b \to +\infty\),
  \[
    \lambda_1(b) = |a|\Theta_0 b + \ordo(b).
  \]

  If $a \geq \Theta _ 0 ^ {-1}$ then, as \(b \to +\infty\),
  \[
    \lambda_1(b) = b + \ordo(b).
  \]
\end{proposition}

\begin{remark}[Continuity of the transitions]\label{rem:discussion-beta a}
The transition between the three regimes in
Proposition~\ref{prop:criticaleig} is continuous with respect to \(a\), due to
the properties of the spectral constant \(\hat\beta_a\) established in
Appendix~\ref{sec:ms}.
 In particular, the function
\(a\mapsto\hat\beta_a\) is continuous and
\[
\hat\beta_{a}=\Theta_0=|a|\Theta_0, \qquad \text {at } a = -1.
\]
Hence, the leading-order asymptotics match continuously at the transition points
$a=-1$. It is easy to see that there is a continuous transition also at $a=
\Theta_0^{-1}$.
\end{remark}

\begin{remark}[Proposition \ref{prop:subcriticaleig} and Proposition \ref{prop:criticaleig}]\label{rem:a tends to 0}
  The regime discussed in Section~\ref{sec:subcritical} corresponds formally to
  taking $|a|=b^{\varrho-1}\ll 1$. In this case, for both $a = \pm
  b^{\varrho-1}$, we have $-1<a<\Theta_0^{-1}$ and the asymptotics in
  Proposition~\ref{prop:criticaleig} are consistent with those obtained in
  Proposition \ref{prop:subcriticaleig}.
\end{remark}

\begin{proof}[Sketch of Proof of Proposition~\ref{prop:criticaleig}] The proof of Proposition~\ref{prop:criticaleig} follows the same scheme as for the
case of constant magnetic field (described in e.g. \cite[Section 8.2]{FH-b}),
with slight adjustments, where needed. We will describe the procedure briefly.
  Loosely speaking, the procedure is as follows: For the lower bound, one uses a
  partition of unity to compare with local model operators. In our setting, these
  will be the de Gennes model (Appendix~\ref{sec:dG}) or the model with
  discontinuous potential (Appendix~\ref{sec:ms}). For the upper bound, one
  constructs trial states localized in small regions, using Gaussians or the
  ground states of the models in Appendix~\ref{sec:pre}.

  With the help of a partition of unity, we will encounter working with functions
  \(u\) that have support in smaller domains \(\Omega_i\), $i=1,2,3$, belonging to one of the
  following three types:
  \begin{itemize}
    \item
      $\Omega_1 = \set{x\in\overline{\Outer}\colon \dist(x,\partial\Inner)>b^{-3/8}}$;
    \item
      $\Omega_2 = \set{x\in\Inner\colon \dist(x,\partial\Inner)>b^{-3/8}}$;
    \item
      $\Omega_3 = \set{x\in\Domain\colon \dist(x,\partial\Inner)<2b^{-3/8}}$.
  \end{itemize}
  The first type, $\Omega_1$, is completely in the closure of the outer domain \(\Outer\), the second, $\Omega_2$, in the
  inner \(\Inner\) and the third, $\Omega_3$, in both $\Outer$ and $\Inner$. Additionally, for the first two types,
  \(u\) will vanish on \(\{\dist(x,\partial\Inner)=b^{-3/8}\}\). By
  \cite[Proposition~8.2.2 and Lemma~1.4.1]{FH-b}, we get
  \[
    \norm {(-\ii\nabla - \Ab)u} _ {L ^ 2(\Omega)} ^ 2
    \geq
    \begin{cases}
      \bigl(\abs {a}\Theta_0 b+o(b)\bigr)\norm {u} _ {L ^ 2(\Omega)} ^ 2
      & \text{ in }\quad \Omega_1 \\
      b \norm {u} _ {L ^ 2(\Omega)} ^ 2
      & \text{ in }\quad \Omega_2
    \end{cases}
  \]
  For the domain \(\Omega_3\), we  compare with the model in Appendix~\ref{sec:ms}.
  We omit the details which can be found in~\cite[Section~4]{A}, and just mention
  the resulting lower bound,
  \[
    \norm {(-\ii\nabla - \Ab)u} _ {L ^ 2(\Omega)} ^ 2
    \geq
    \bigl(\hat{\beta} _ a b + o(b) \bigr)\norm {u} _ {L ^ 2(\Omega)} ^ 2.
  \]

  Summing up, and noticing that the errors introduced by the partition of
  unity are small in comparison, which can be proved as in \cite[Section 8.2]{FH-b} and \cite[Section~4]{A}, we get
  \[
    \lambda_1(b)\geq \min(|a|\Theta_0,1,\hat\beta_a)b+o(b)
  \]
  as required. Note that, by the properties mentioned in Appendix~\ref{sec:ms},
  $\hat\beta_a=\min(a,1)$  for $a>0$, and $\hat\beta_a> |a|\Theta_0$ for $-1<a<0$.

  An upper bound can be obtained by considering a trial state:

  \begin{itemize}
    \item
      The same trial state in \eqref{eq:trial-state} for $|a|<1$;
    \item
      A trial state like in \eqref{eq:trial-state} but adjusted by replacing
      $\sigma$ with $b$, for $a>1$;
    \item
      A trial state like in \eqref{eq:trial-state} but adjusted by replacing
      $\sigma$ with $b$, and by replacing $(\xi_0,\varphi_0)$ with the
      corresponding configuration for the model in Appendix~\ref{sec:ms}, for
      $a<-1$. Again, we omit the details and refer to \cite[Section~4]{A}.
  \end{itemize}

  The upper bounds one gets match the previously stated lower bounds.
\end{proof}

Given these asymptotic results on the eigenvalues, we again obtain Agmon
estimates. In this case we need to divide the situation into three cases,
depending on the relative sizes of the involved constants.

\begin{lemma}\label{lem:criticalagmon}
  Suppose that $a \neq 0$ is fixed.
  \begin{enumerate}
   \item
    If $a<-1$, there exist positive constants $C_1,b_1$ and $\delta_1<1$ such
    that every normalized ground state $u$ of $\lambda_1(b)$ satisfies
    for all $b\geq b_1$,
    \[
      \int_{\Domain}
      \Bigl(b^{-1}|(-\ii\nabla-\Ab)u|^2+|u|^2 \Bigr)
      \ee^{\delta_1 b^{1/2}\dist(x,\partial\Inner)}\dd x
      \leq C_1.
    \]
\item
    If $-1< a< \Theta_0^{-1}$, there exist positive constants $C_2,b_2$ and
    $\delta_2 < 1$ such that every normalized ground state $u$ of
    $\lambda_1(b)$ satisfies, for all $b\geq b_2$,
    \[
      \int_{\Domain}
      \Bigl(b^{-1}|(-\ii\nabla-\Ab)u|^2+|u|^2 \Bigr)
      \ee^{\delta_2 b^{1/2}\dist(x,\partial\Domain)}\dd x
      \leq C_2.
    \]
    \item
    If $a>\Theta_0^{-1}$, there exist positive constants $C_3,b_3$ and
    $\delta_3<1$ such that every normalized ground state $u$ of
    $\lambda_1^{\Domain}(b)$ satisfies for all $b\geq b_3$,
    \[
      \int_{\Outer}
      \Bigl(b^{-1}|(-\ii\nabla-\Ab)u|^2+|u|^2 \Bigr)
      \ee^{\delta_3 b^{1/2}\dist(x,\partial\Inner)}\dd x
      \leq C_3.
    \]
\end{enumerate}
\end{lemma}
We omit the proof, since it follows exactly a standard scheme \cite[Section~8.2.3]{FH-b}.
Lemma~\ref{lem:criticalagmon} demonstrates that the values $a = -1$ and $a =
1/\Theta _ 0$ are also critical in the sense that the concentration of the ground
state switches from one region to another. Notice that when $-1<a<\Theta_0^{-1}$, the sign of $a$ does not affect the localization region described in Lemma~\ref{lem:criticalagmon}(ii); both positive and negative values in this interval lead to concentration near the outer boundary $\partial\Domain$. In fact, this is the reason behind the possible inclusion of $-1<a<0$ in the statement of Theorem~\ref{thm:maincritical}.

\subsection{The critical regime for the unit disk (\texorpdfstring{$-1 < a < \Theta_0^{-1}$}{-1<a<1/Theta0})}\label{subsec:criticaldisk}

In the case $-1<a<\Theta_0^{-1}$, we have localization near the boundary of $\Domain$
and an additional flux term coming from the interior of $\Inner$,
\[
  \Phi=\frac{ab\measure{\Outer}}{2\pi}+\frac{b\measure{\Inner}}{2\pi}.
\]
In the case of the disk, we expect this extra flux to contribute to the ground
state energy asymptotics as in the subcritical regime
(Proposition~\ref{prop:disk-subcritical}).

We write an accurate asymptotic expansion for $\lambda_1(b)$ that involves the constants
$\Theta_0$ and $\varphi_0(0)$ from Appendix~\ref{sec:dG}.

\begin{proposition}\label{prop:disk-critical}
  Suppose that \(\Domain\) is the unit disk. Let \(a \neq 0\) be fixed with
  \(-1< a<\Theta_0^{-1}\). Let \(\Phi\) be the flux of the field \(B\) through
  \(\Domain\). There exist constants \(C_0\) and \(C_1\)
  such that if
  \[
    \Delta_b
    \coloneqq
    \inf_{m \in \integers}
    \left|m - \Flux - (|a|b)^{1/2}\xi_0 - C_0 \right|
  \]
  then, as $b\to+\infty$, the lowest eigenvalue satisfies
  \[
    \lambda_1(b)
    =
    \abs {a}\Theta_0b
    - \frac13\varphi_{0}(0)^2(|a|b)^{1/2}
    + \xi_0\varphi_{0}(0)^2\bigl((\Delta_b)^2 + C_1 \bigr)
    + \Ordo(b^{-1/2}).
  \]
\end{proposition}

Note that if \(\Inner\) and \(\Domain\) are concentric disks, then the magnetic field is
radial, and the result is known from~\cite{FS}. Proposition~\ref{prop:disk-critical}
can therefore be viewed as a generalization of the result in~\cite{FS}.
However, the proof in~\cite{FS} only uses the radial structure of the magnetic
field near the boundary of \(\Domain\), which is precisely the situation covered by
Proposition~\ref{prop:disk-critical}. For this reason, we omit the proof and
refer the reader to~\cite{FS} (see also the proof of
Proposition~\ref{prop:disk-subcritical}).

\subsection{The critical regime for generic domains (\texorpdfstring{$-1 < a < \Theta_0^{-1}$}{-1<a<1/Theta0})}\label{subsec:critical-monotonicity}

In this regime, we prove that the one-sided derivative, $\lambda_{1,+}'(b)$, of the ground state energy, $\lambda_1(b)$, is positive for large $b$.  

\begin{proposition}\label{prop:critical-generic-monotone}
  Suppose that $\Domain\subset \mathbb{R}^2$ is not a disk. Let $-1<a<\Theta_0^{-1}$ and let the magnetic field be as in \eqref{eq: mag field crit}. 
  Then
  \[
    \liminf_{b \to +\infty}\lambda_{1,+}'(b) \geq \Theta_0|a| > 0.
  \]
\end{proposition}

\begin{proof}[Sketch of the proof] Since $-1<a<\Theta_0^{-1}$, the ground states localize near $\partial\Domain$ by Lemma~\ref{lem:criticalagmon}. This allows us to derive a two-term asymptotic expansion
\[
\lambda_1(b)=\Theta_0|a| b - \frac13\varphi_0(0)^2\kappa_{\max}|a|^{1/2}b^{1/2} + \ordo(b^{1/2}),
\]
and the ground states are localized near the set $\{\kappa=\kappa_{\max}\}$ where the curvature is maximal on $\partial\Domain$. The proof then follows verbatim from that of Proposition~\ref{prop:sub-generic-monotone}.
\end{proof}

\subsection{Proof of Theorem~\ref{thm:maincritical}}

\begin{proof}[Proof of Theorem~\ref{thm:maincritical}]
We first discuss the case  $-1<a < \Theta_0^{-1}$ for the unit disk and for general domains and later the case $a> \Theta_0^{-1}$. As noted earlier, we will prove the theorem allowing negative $a$.
\paragraph{Case $-1 < a < \Theta_0^{-1}$.}~\medskip

\noindent\textbf{Unit disk.} For the unit disk, as in \cite[Theorem~1.7]{FS}, loss of monotonicity results from a competition between the leading term in $\lambda_1(b)$ and the oscillatory term $\Delta_b$. Note that, unlike the subcritical regime, the derivative with respect to $b$ of the leading term is non-zero.  As an immediate consequence of Proposition~\ref{prop:disk-critical}, one can prove that for $-1 < a < \Theta_0^{-1}$ and $a\neq 0$, defining
\[
    \Phi_0=\frac{a|\Outer|}{2\pi}+\frac{|\Inner|}{2\pi},
\]
it holds that
\begin{itemize}
\item If $\Phi_0<\frac{|a|\Theta_0 }{\xi_0\varphi_{0}(0)^2},$ then there exists $b_0>0$ such that the function $b\mapsto\lambda_1(b)$ is monotone increasing on $[b_0,+\infty)$.
\item If $\Phi_0>\frac{|a|\Theta_0 }{\xi_0\varphi_{0}(0)^2},$ then for all $b_0>0,$ there exist $b_1$ and $b_2$ with $b_0<b_1<b_2$ and
\[\lambda_1(b_0)<\lambda_1(b_1),\quad\lambda_1(b_2)<\lambda_1(b_1).\]
\end{itemize}
For the proof, we refer to \cite[Section 6, Proof of Theorem~1.7]{FS}.\medskip

\noindent\textbf{General (non-disk) domains.}
We suppose now that $\Domain$ is not a disk. Thanks to Proposition~\ref{prop:critical-generic-monotone}, there exists $b_0>0$ such that the function $b\mapsto\lambda_1(b)$ is monotone increasing on $[b_0,+\infty)$.

\paragraph{Case $a > \Theta_0^{-1}$.}

The case $a>\Theta_0^{-1}$ resembles the Dirichlet Laplacian with constant magnetic field in $\Inner$. To leading order, the asymptotics for $\lambda_1(b)$ agrees with that of the Dirichlet magnetic Laplacian (see Proposition~\ref{prop:criticaleig}).
For general domains including disks, we will prove that there exists $b_0>0$ such that the function
  $b\mapsto\lambda_1(b)$ is increasing on $[b_0,+\infty)$.

    Let us write $g(b)=\lambda_1(b)-b$. By Proposition~2.2 in \cite{FH}, if
    \[
      \text{for all $\epsilon\in(-1,1)$, $\quad\abs {g(b+\epsilon)-g(b)}\to0$
      as $b\to+\infty$,}
    \]
    then the function $b\mapsto \lambda_1(b)$ is monotone increasing on some
    interval $[b_0,+\infty)$.

    It suffices to prove that $g(b)=o(1)$ as $b\to+\infty$. Let $\lambda_1 ^
    {\Inner}(b)$ be the lowest eigenvalue of the Laplacian with constant magnetic
    field in $\Inner$ and Dirichlet boundary condition on $\partial \Inner$. By the
    min-max principle, we have
    \[
      \lambda_1(b)\leq \lambda_1^\Inner(b)
    \]
    and by \cite[Proposition~4.1]{HS},
    \[
      \lambda_1^\Inner(b)= b+\Ordo(b^{-\infty}).
    \]

    Let $\lambda_1^\Domain(b)$ be the lowest eigenvalue of the Laplacian with magnetic field $B$ in $\Domain$ and Dirichlet boundary condition on $\partial \Inner$. By Lemma~\ref{lem:criticalagmon}, we have
    \[
      \lambda_1(b)\geq \lambda_1^\Domain(b)+\Ordo(b^{-\infty}),
    \]
    and by \cite[Lemma~1.4.1]{FH-b} (see also \cite[Proposition~5.1]{KW} for the
    case of square integrable $B$)
    \[
      \lambda_1^\Domain(b)\geq \inf B.
    \]
     Since \(a>1\), we have \(\inf B=b\). Thus, collecting the upper and lower bounds, we get \( \lambda_1(b)= b+\Ordo(b^{-\infty})\). Consequently,  \(g(b)=o(1)\) as \(b\to+\infty\). This finishes the proof of Theorem~\ref{thm:maincritical} in the case $a>\Theta_0^{-1}$.
\end{proof}

\begin{remark}[The critical regime with $a<-1$]\label{rem:crit-a<-1}
We  expect monotonicity  in the case where $a<-1$, since the ground states are concentrated near the boundary of $\Inner$ (see Lemma~\ref{lem:criticalagmon}). Since the topology of $\Inner$ is trivial, the flux is simply the applied magnetic field in $\Inner$ without additional geometric contributions. Although we do not discuss the case $a<-1$ in detail, we highlight some technical aspects:
\begin{enumerate}
    \item For generic non-disk domains, we anticipate that proving the monotonicity of the function $b\mapsto \lambda_1(b)$ on some interval $[b_0,+\infty)$ is possible. As we did in the proof of Proposition~\ref{prop:sub-generic-monotone}, one only needs the localization properties of the ground states near the boundary points of $\Inner$ with maximal curvature, and the leading order asymptotics of $\lambda_1(b)$. These properties were established in \cite{AK}.
    \item For the disk, we need the asymptotic expansion of the ground state energy with three terms, like the one in Proposition~\ref{prop:disk-critical}. While this formula is not mentioned in the literature for our specific magnetic field, it can be obtained by using the results in \cite{AK} and the methods in \cite[Appendix C]{FH} (note that the first two terms were given in \cite{AK} and only the third term---the oscillatory term---is missing). With this formula in hand, one can reduce the problem of monotonicity to the positivity of a certain coefficient, exactly as in \cite[Proposition~2.7]{FH}.
\end{enumerate}
\end{remark}

\section{The supercritical regime}\label{sec:supercritical}
In this section, we analyze the supercritical regime, namely the case in which the magnetic field $\OuterB(b)$ in the outer region \(\Outer\) grows faster than the field $b$ in the inner region \(\Inner\). More precisely, we consider, for some \(\varrho > 1\), a magnetic field of the form 
\begin{equation}\label{eq: mag field super crit}
B(x)=\begin{cases}
    b&\text{in $\Inner$},\\
     b^{\varrho}&\text{in $\Outer$}.
\end{cases}\end{equation}
The main result of this section is the proof of Theorem~\ref{thm:supercritical}. 
Heuristically, the strong magnetic field in \(\Outer\) acts as a confining barrier, energetically penalizing low-energy states in the outer region and forcing the eigenfunctions to concentrate inside \(\Inner\). In the large field limit, the problem therefore behaves like the magnetic Laplacian in the inner domain \(\Inner\) with Dirichlet condition on \(\partial \Inner\). Consequently, one would expect the behavior of the lowest eigenvalue to be entirely governed by the operator in $\Inner$, while oscillatory flux effects originating from the outer region would disappear. Hence, strong diamagnetism is expected independently of the geometry of $\Domain$.

Although Theorem~\ref{thm:supercritical} only applies to magnetic fields of the form \eqref{eq: mag field super crit}, some of the results established in this section remain valid for magnetic field intensities in \(\Outer\) of the form
\[
\OuterB(b)=\pm b^\varrho .
\]
Thus, part of our analysis also covers the case of negative magnetic fields in the outer region \(\Outer\).

We start with an asymptotic result for the lowest eigenvalue.

\begin{proposition}\label{prop:supercriticaleig}
  Suppose that $\varrho>1$ and $|\OuterB(b)|=b^{\varrho}$. The lowest eigenvalue
  $\lambda_1(b)$ satisfies, as $b\to+\infty$,
  \[
    \lambda_1(b)=b+o(b).
  \]
\end{proposition}
Before presenting the proof of Proposition~\ref{prop:supercriticaleig}, we give two remarks.

\begin{remark}[Consistency with the critical regime]\label{rem:outereig}
The regime considered in Proposition~\ref{prop:supercriticaleig} corresponds to the critical regime in Proposition~\ref{prop:criticaleig} with $a=\pm b^{\varrho-1}$. Since $|a|\gg 1$, the asymptotics in Proposition~\ref{prop:supercriticaleig} is consistent with those in Proposition~\ref{prop:criticaleig} for large $|a|$
(by Proposition~\ref{prop:beta-a}, $\hat\beta_a\to 1$ as $a\to-\infty$).
\end{remark}
\begin{remark}[Fixed magnetic field in the inner domain]\label{rem:DirichletLap}
Suppose that the magnetic field $B$ is strong in $\Outer$, equal to $b\gg 1$, and fixed in $\Inner$, equal to a constant $a$. Then, arguing as in Proposition~\ref{prop:asyequi}, we obtain that for any fixed $n\in\naturalnumbers$,
\[\lambda_n(b)=\lambda_n^{\Inner}(a)+o(1),\]
where $\lambda_n^{\Inner}(a)$ is the $n$th eigenvalue (counting multiplicity) of the Laplacian with magnetic field $a$ in $\Inner$ and Dirichlet boundary condition.
\end{remark}

\begin{proof}[Proof of Proposition~\ref{prop:supercriticaleig}]
The upper bound can be obtained easily by comparing with the lowest eigenvalue
$\lambda_1^{\Inner}(b)$ of the Dirichlet Laplacian with constant magnetic field
$b$ in $\Inner$,
\[
  \lambda_1(b)\leq \lambda_1^{\Inner}(b)=b+o(b).
\]
For the lower bound, we estimate the quadratic form $\Quadraticform$ introduced in \eqref{eq:def-qf} from below. When integrating over a subset $D\subset\Omega$ instead of $\Omega$, we denote the resulting integral by $\Quadraticform[\psi,D]$. By means of a partition of unity, we reduce to bounding $\Quadraticform[\psi,D_i]$ from below for $i=1,2,3$ with
\begin{itemize}
  \item  $D_1=\{x\in\Outer\colon   \dist(x,\partial\Inner)>b^{-3/8}\}$;
  \item $D_2=\{x\in\Inner\colon \dist(x,\partial\Inner)>b^{-3/8}\}$;
  \item $D_3=\{x\in\Domain\colon \dist(x,\partial\Inner)<2b^{-3/8}\}$. 
\end{itemize}

In $D_1$ and $D_2$, we assume that $\psi$ vanishes on
$\{\dist(x,\partial\Inner)=b^{-3/8}\}$, while in $D_3$, we assume that 
$\psi$ is compactly supported in $D_3$. As explained in the critical regime, we still
have
\[
  \Quadraticform[\psi,D_i]
  \geq
  \begin{cases}
    \bigl(\Theta_0 \abs {\sigma} +o(|\sigma|)\bigr)\int _ {D_i} \abs {\psi} ^ 2\dd x
    &\text{ for } i=1\\
    b \int_{D_i} \abs {\psi} ^ 2\dd x
    &\text{ for } i=2.
  \end{cases}
\]
In the third region, we compare with the model in Appendix~\ref{sec:ms} as done in
\cite[Section~4]{A}, and we have
\[
  \Quadraticform[\psi,D_3]\geq \bigl(\hat\beta_z b+o(b) \bigr)\int_{D_3}|\psi|^2\dd x,
\]
where $z= \sigma b^{-1}= \pm b^{\varrho-1}$, and $\hat\beta_z$ is introduced in \eqref{eq:def-hat-beta}.

Knowing that $|\sigma|\gg b$, we get after collecting all the aforementioned estimates,
\[
  \lambda_1(b)\geq \min(1,\hat\beta_z)b+o(b).
\]
For $\sigma>0$, we have seen in Appendix~\ref{sec:ms} that $\hat\beta_z=1$. For
$\sigma<0$, $z\to-\infty$ since $\varrho>1$, and we have by
Proposition~\ref{prop:beta-a}, $\hat\beta_z\to 1$. Therefore
\[
  \lambda_1(b)\geq b+o(b)
\]
as required.
\end{proof}

With the asymptotics in Proposition~\ref{prop:supercriticaleig}, we get by Agmon estimates that the ground state decays in $\Outer$.

\begin{lemma}\label{lem:superagmon}
  Suppose that $\varrho>1$ is fixed and $|\OuterB(b)|=b^{\varrho}$. There exist positive constants $C_2,b_2$ and $\delta_2<1$
    such that every normalized ground state $u$ of $\lambda_1(b)$
    satisfies for all $b\geq b_2$,
    \[
      \int_{\Outer}
        \Bigl(b^{-\varrho}|(-\ii\nabla-\Ab)u|^2+|u|^2 \Bigr)
        \ee^{\delta_2 b^{\varrho/2}\dist(x,\partial\Inner)}\dd x
      \leq C_2.
    \]
\end{lemma}
We omit the proof, since it follows exactly a standard scheme \cite[Section~8.2.3]{FH-b}.
Now we turn to the question of monotonicity of $\lambda_1(b)$ with respect to $b$.

\begin{proof}[Proof of Theorem~\ref{thm:supercritical}]
Suppose now that $\sigma(b)=b^{\varrho}$, hence positive. Thanks to Lemma~\ref{lem:superagmon}, starting from the estimate in Proposition~\ref{prop:supercriticaleig}, we get
\[
    \lambda_1(b)=b+\Ordo(b^{-\infty}).
\]
This follows by the same scheme as the proof of Theorem~\ref{thm:maincritical} in the case $a>\Theta_0^{-1}$.
Thus, by  \cite[Proposition 2.2]{FH}, there exists $b_0>0$ such that the function $b\mapsto \lambda_1(b)$ is monotone increasing on $[b_0,+\infty)$.
\end{proof}

\appendix

\section{Two technical proofs}\label{app:proofs}

We present here the proofs of Propositions~\ref{prop:conv-ef} and \ref{prop:ev-high-N} that hold for the perfect oscillatory regime.

\subsection{Proof of Proposition~\ref{prop:conv-ef}}

In the hypotheses of Proposition~\ref{prop:conv-ef},  $\beta_*$ is fixed in the base interval and the sequence $b_k$ is as in \eqref{eq:def-bk}.  Note that, using Proposition~\ref{prop:asyequi},  $u^{(j)}_k$ is such that
\[
    \Hamiltonian(b_k) u^{(j)}_k = \lambda_j(b_k)  u^{(j)}_k = \lambda_j^{\Outer}(b_k)  u^{(j)}_k + o(1),\qquad \forall j=1, \ldots,n,
\]
Using Lemma \ref{lem:HHOO} and the definition of $w_k^{(j)}$ together with \eqref{eq:def-bk}, we can deduce that
\begin{equation}\label{eq:ev-eq-limit}
  \Hamiltonian ^ {\Outer}(\beta_\ast)w_k^{(j)}
  =
  \lambda_j(b_k)w_k^{(j)}\quad\text{in $\Outer$}, \qquad \forall \quad j=1, \ldots , n
\end{equation}
with
\[
  \Hamiltonian ^ {\Outer}(\beta_*)
  =
  (-\ii\nabla-\Ab_*)^2,\quad
  \Ab_*
  \coloneq
  \beta_*\Fb_{\Inner}+a\Fb_{\Outer}.
\]
By elliptic estimates and Lemma~\ref{lem:asyequi-H1}, we deduce that, for every
$\varepsilon\in(0,\varepsilon_0)$, $w_k^{(j)}$ is bounded in
$H^2(\Outer_\varepsilon)$, where $\Outer_\varepsilon$ was introduced in
Notation~\ref{notation}. By Lemma~\ref{lem:asyequi-L2} and a Cantor diagonal
sequence argument, we obtain a subsequence (not relabeled) such that, for every
$\varepsilon\in(0,\varepsilon_0)$, we have
\[
  w_k^{(j)}\rightharpoonup w_*^{(j)}\quad\text{in $H^2(\Outer_\varepsilon)$},\quad w_k^{(j)}\to w_*^{(j)}\quad\text{in $H^1(\Outer_\varepsilon)$},
\]
and $w_k^{(j)}\to w_*^{(j)}$ in $L^2(\Outer)$. Note that $w_*^{(j)}$ satisfies Neumann boundary condition on the outer boundary of $\Outer$; it inherits this condition from $w_k^{(j)}$ and the convergence in $H^2(\Outer_\varepsilon)$. Moreover, by Lemma~\ref{lem:asyequi-H1}, $w_*^{(j)}\in
H^1(\Outer)$ and $w_*^{(j)}|_{\partial \Inner}=0$.
Thus, by \eqref{eq:ev-eq-limit} and Proposition~\ref{prop:asyequi}, $w_*^{(j)}$ is a weak solution of 
\[\begin{cases}
     (-\ii\nabla-\Ab_*)^2w_*^{(j)}=\lambda_j^\Outer(\beta_*)w_*^{(j)}&\text{in $\Outer$,}\\
     w_*^{(j)}=0&\text{on the inner boundary of $\Outer$,}\\
     \nu\cdot\nabla w_*^{(j)}=0&\text{on the outer boundary of $\Outer$.}
\end{cases}\]
By Lemma~\ref{lem:asyequi-ON}, $\{w_*^{(j)}\}_{1\leq j\leq n}$
is orthonormal in $L^2(\Outer)$. Hence $w_*^{(j)}$ is an eigenfunction of $\Hamiltonian^\Outer(\beta_*)$. 

Using the triangle inequality, we write
  \[
    \norm {\nabla w_{k}^{(j)}- \nabla w_{*}^{(j)}}_{L^2(\Outer;\complexes^2)}\leq I_1(k)+I_2(k),
  \]
  where
  \[
    I_1(k)
    =
    \norm {(-\ii\nabla - \Ab_*)(w_{k}^{(j)}-w_{*}^{(j)}) } _ {L ^ 2(\Outer)},
    \quad
    I_2(k) = \norm {\Ab_*(w_{k}^{(j)}-w_{*}^{(j)})} _ {L ^ 2(\Outer;\complexes^2)}.
  \]
  Clearly, $I_2(k)\to0$ as $k\to+\infty$, by the convergence in $L^2(\Outer)$.
  Concerning $I_1(k)$, we use the polarization identity and integration by
  parts\footnote{Recall that $w_{*}^{(j)}|_{\partial\Inner}=0$ and $w_{k}^{(j)}$
  satisfies Neumann boundary condition on $\partial\Domain$.},
  \begin{equation*}
    \begin{aligned}
        I_1(k)^2
        &
        =
        \|(-\ii\nabla -\Ab_*)w_{k}^{(j)}\|_{L^2(\Outer)}^2
        + \|(-\ii\nabla - \Ab_*)w_{*}^{(j)}\|_{L^2(\Outer)}^2
        \\
        &\qquad
        -2\re
        \langle
        (-\ii\nabla - \Ab_*)^2w_{k}^{(j)},w_{*}^{(j)}\rangle_{L^2(\Outer)}.
    \end{aligned}
  \end{equation*}
We observe the following:
\begin{enumerate}
  \item By Lemma~\ref{lem:HHOO},
    \[
      \norm {(-\ii\nabla -\Ab_*)w_{k}^{(j)}} _ {L^2(\Outer)} ^ 2
      =
      \norm {(-\ii\nabla-\Ab)u_k^{(j)}} _ {L^2(\Outer)} ^ 2
      \leq
      \lambda_j(b_k).
    \]
    \item By \eqref{eq:goal-rayleigh}, Proposition~\ref{prop:asyequi}, and Lemma \ref{lem:HHOO},
    \[
      \norm {(-\ii\nabla - \Ab_*)w_{*}^{(j)}} _ {L^2(\Outer)} ^ 2
      \leq
      \lambda_j^\Outer(\beta_*).
    \]
  \item By \eqref{eq:ev-eq-limit},
    \[ \innerproduct { (-\ii\nabla - \Ab_*)^2w_{k}^{(j)},w_{*}^{(j)} } _ {L^2(\Outer)}
      =
      \lambda_j(b_k)
      \innerproduct { w_{k}^{(j)},w_{*}^{(j)} } _ {L ^ 2(\Outer)}.
    \]
\end{enumerate}
Thus, by Proposition~\ref{prop:asyequi} and Lemma~\ref{lem:asyequi-L2}, we obtain that
\(I_1(k)\to0\) as $k\to+\infty$, and consequently that $w_k^{(j)}\to w_*^{(j)}$ in $H^1(\Outer)$.

Now that we have proved \eqref{eq: conv min in H1omega0}, we have by \cite[Theorem 1]{MM} that $|w_{k}^{(j)}|\to |w_*^{(j)}|$ in $H^1(\Outer)$, thereby proving \eqref{eq: conv mod H1}. The last formula in Proposition~\ref{prop:conv-ef} follows immediately from \eqref{eq:ev-eq-limit} and the convergence of $w_k^{(j)}$.

\subsection{Proof of Proposition~\ref{prop:ev-high-N}}
    We split the proof in several steps.

  \noindent\textbf{Step 1} (Lowest eigenvalue in $\Inner$)\textbf{.}

  For any $0 <\lambda< \Theta_0$, the restriction on $\lambda$ ensures that there are no eigenvalues below
  $\lambda b$ for the Neumann or Dirichlet realizations of $(-\ii\nabla-\Ab)^2$
  in $L^2(\Inner)$, for $b$ sufficiently large. In fact, since $\curl\Ab=b$ in
  $\Inner$, the lowest eigenvalue for the Neumann realization in $\Inner$  is
  asymptotic to $\Theta_0 b$ as $b\to+\infty$ (see
  \cite[Thm.~8.1.1]{FH-b}).\medskip

\noindent\textbf{Step 2} (Dirichlet--Neumann bracketing)\textbf{.}

For a self-adjoint operator $T$ with pure discrete spectrum and for $\Lambda\in\reals$, we let
$N(T,\Lambda)=\dim(\mathbf{1}_{(-\infty,\Lambda)}(T))$, the number of
eigenvalues of $T$ below $\Lambda$ (counting multiplicities).

For a given open set $\omega\subset \Domain$, let $\Hamiltonian^{\omega,D}(b)$
(respectively $\Hamiltonian ^ {\omega,N}(b)$) be the Dirichlet realization
(respectively Neumann realization) of $\Hamiltonian (b)$ in the domain $\omega$.
Thanks to the identification \(L^2(\Domain)\cong L^2(\Inner)\oplus L^2(\Outer)\),
we consider the operators
\(\Hamiltonian^{\Inner,D}(b)\oplus\Hamiltonian^{\Outer,D}(b)\) and
\(\Hamiltonian^{\Inner,N}(b)\oplus\Hamiltonian^{\Outer,N}(b)\) as self-adjoint
operators in $L^2(\Domain)$. Using the obvious decomposition of the quadratic
form
\[
  \Quadraticform[\psi]
  =
  \Quadraticform^{\Inner}[\psi]
  +
  \Quadraticform^{\Outer}[\psi]
\]
and the comparison of form domains, we deduce that
\[
  \Hamiltonian ^ {\Inner,N}(b) \oplus \Hamiltonian ^ {\Outer,N}(b)
  \leq
  \Hamiltonian (b)
  \leq
  \Hamiltonian ^ {\Inner,D}(b) \oplus \Hamiltonian^{\Outer,D}(b).
\]
Thus, by the min-max principle,
\[
  N(\Hamiltonian ^ {\Inner,D}(b),\Lambda)
  +
  N(\Hamiltonian ^ {\Outer,D}(b),\Lambda)
  \leq
  N(\Hamiltonian (b),\Lambda)
  \leq
  N(\Hamiltonian ^ {\Inner,N}(b),\Lambda)
  +
  N(\Hamiltonian ^ {\Outer,N}(b),\Lambda).
\]
Suppose now that $1<\Lambda<\Theta_0 b$. By Step 1, we get that $N(\Hamiltonian ^
{\Inner,D}(b),\Lambda)=N(\Hamiltonian ^{\Inner,N}(b),\Lambda)=0.$ Thus,
\[
  N(\Hamiltonian ^{\Outer,D}(b),\Lambda)
  \leq
  N(\Hamiltonian (b),\Lambda)
  \leq
  N(\Hamiltonian ^{\Outer,N}(b),\Lambda).
\]
Choose an integer \(p\) such that \(\beta=b-2\pi p/\measure {\Inner}\) belongs to
the base interval $[0,2\pi/\measure {\Inner})$. By Lemma~\ref{lem:HHOO},
$\Hamiltonian ^ {\Outer,N}(b)$ and $\Hamiltonian ^ {\Outer,D}(b)$ are unitarily equivalent to
$\Hamiltonian ^ {\Outer,N}(\beta)$ and $\Hamiltonian ^ {\Outer,D}(\beta)$ respectively.
Moreover, for every $\delta>0$, there exists $C>0$ and we have the quadratic
form comparison
\[
  (1-\delta)\int_{\Outer}\abs {\nabla \psi} ^ 2\dd x - \delta^{-1}C
  \leq
  \int_{\Outer}\abs {(-\ii\nabla-\beta\Fb_\Inner-a{\Fb_\Outer})\psi}^2\dd x
  \leq (1+\delta)\int_{\Outer}\abs {\nabla \psi} ^ 2\dd x + \delta^{-1}C.
\]
Consequently, choosing $\delta=\Lambda^{-1/2}$ and using the min-max principle, we obtain that
\begin{equation}\label{eq:bracketing}
  N(-\Delta^{\Outer,D},\Lambda-(2+C)\Lambda^{1/2})
  \leq
  N(\Hamiltonian (b),\Lambda)
  \leq N(-\Delta^{\Outer,N},\Lambda+(2+C)\Lambda^{1/2}).
\end{equation}
Now, taking $\Lambda\gg 1$ but still $\Lambda < \Theta_0b$, we have the Weyl law for the Dirichlet/Neumann Laplacian in $\Outer$ (see \cite[Chapter 3, p. 168]{FLW})
\begin{equation}\label{eq:Weyl-Lap}N(-\Delta^{\Outer,\#},\Lambda\pm (2+C)\Lambda^{1/2})\sim \bigl(\Lambda\pm (2+C)\Lambda^{1/2}\bigr)\frac{|\Outer|}{4\pi}\quad(\#\in\{N,D\}).
\end{equation}

\noindent\textbf{Step 3} (Comparison of eigenvalues and Weyl laws)\textbf{.}

  First, by Proposition \ref{prop:asyequi}, and in particular \eqref{eq:simplebounds}, we have
  \[\lambda_n(b)\leq \lambda_n^\Outer(b)\quad\text{for all $n\in\mathbb N$.}\]
  We argue by contradiction. Suppose \eqref{eq: weyl n}  is false: $\sup_{n\in \mathcal{I}_\lambda(b)}|\lambda_n(b)/\lambda_n^{\Outer}(b) -1|$ does not converge to $0$. Then, there exists a non-negative constant $c_*<1$ and
  sequences\footnote{$n_j\to+\infty$ in light of
  Proposition~\ref{prop:asyequi}}
  \[
    n_j \to +\infty,\quad b_j\to+\infty,
  \]
  such that $n_j \in\mathcal{I}_\lambda(b_j)$, i.e.,
  \[
    \lambda_{n_j}^{\Outer}(b_j)<\lambda b_j \text{ for all $j$},
  \]
  and
  \[
    \lim_{j\to+\infty}\frac{\lambda_{n_j}(b_j)}{\lambda_{n_j}^{\Outer}(b_j)}
    = c_*.
  \]
  By \eqref{eq:bracketing} and the Weyl law for the Laplacian \eqref{eq:Weyl-Lap}, we have
 \[N(\Hamiltonian (b_j),\lambda_{n_j}(b_j))\underset{j\to+\infty}{\sim} \frac{|\Outer|}{4\pi}\lambda_{n_j}(b_j).\]
 Similarly, arguing like in \eqref{eq:bracketing} but for $\Hamiltonian^\Outer(b)$, the Dirichlet--Neumann bracketing for $\Hamiltonian ^{\Outer}(b)$ yields
 \[  N\bigl(\Hamiltonian ^{\Outer}(b_j),\lambda_{n_j}^{\Outer}(b_j)\bigr)
    \underset{j\to+\infty}{\sim}\frac{|\Outer|}{4\pi}\lambda_{n_j}^{\Outer}(b_j).\]
 We claim that the aforementioned Weyl laws yield
 \begin{equation}\label{eq:claim}
 \lambda_{n_j}(b_j)\underset{j\to+\infty}{\sim}\frac{4\pi}{|\Outer|}n_j,\quad
 \lambda_{n_j}^{\Outer}(b_j)\underset{j\to+\infty}{\sim}\frac{4\pi}{|\Outer|}n_j.
 \end{equation}
 Once the above claim is proved, we obtain the desired contradiction. \medskip

\noindent\textbf{Step 4} (Proof of \eqref{eq:claim})\textbf{.}
The definition of $N(\Hamiltonian (b_j),\lambda_{n_j}(b_j))$ yields
 \[ N(\Hamiltonian (b_j),\lambda_{n_j}(b_j))\geq n_j.\]
 From this, and the Weyl law, we find
 \[ \liminf_{j\to+\infty}\frac{\lambda_{n_j}(b_j)}{n_j}\geq \frac{4\pi}{|\Outer|}.\]
Now, for $\delta\in(0,1)$, we have
\[ N(\Hamiltonian (b_j),(1-\delta)\lambda_{n_j}(b_j))\leq n_j,\]
and by the Weyl law
\[ \limsup_{j\to+\infty}\frac{\lambda_{n_j}(b_j)}{n_j}\leq (1-\delta)^{-1}\frac{4\pi}{|\Outer|}. \]
Since this is true for any $\delta\in(0,1)$, we get after taking $\delta\to0^+$ that
\[ \limsup_{j\to+\infty}\frac{\lambda_{n_j}(b_j)}{n_j}\leq \frac{4\pi}{|\Outer|}.\]
This finishes the proof of the first formula in \eqref{eq:claim}. The second formula can be proved in the same manner.

\begin{remark}\label{rem:ev-high-N}
  If we consider energy levels below $\lambda b$ with $\lambda>\Theta_0$,
  the eigenvalues of $\Hamiltonian ^\Outer(b)$ do not approximate all the corresponding
  eigenvalues of $\Hamiltonian (b)$. In fact, according to Lemma~\ref{lem:HHOO} $\Hamiltonian^\Outer(b)$ is unitarily equivalent to $\Hamiltonian^\Outer(b_*)$ with $b_*$ in the base interval $[0,2\pi/\abs{\Inner})$,   and by the Weyl law for the Laplacian in $\Outer$
  \[
    N(\Hamiltonian ^\Outer(b),\lambda b)=N(\Hamiltonian ^\Outer(b_*),\lambda b)
    \sim \frac{|\Outer|}{4\pi}\lambda b.
  \]
  However,  by Dirichlet--Neumann bracketing, we have
  \[
    N(\Hamiltonian (b),\lambda b)
    \geq  N(\Hamiltonian ^{\Outer}(b),\lambda b)+ N(\Hamiltonian ^{\Inner,D},\lambda b).
  \]
  The counting function of $\Hamiltonian ^{\Inner,D}$ below the energy level $\lambda b$ satisfies \cite[Eq.~(1.3)]{F}
  \[ N(\Hamiltonian ^{\Inner,D},\lambda b)\sim \frac{|\Inner|}{4\pi}\lambda b+o(b)\quad\text{as }b\to+\infty.\]
   Thus, Proposition~\ref{prop:ev-high-N} does not hold for $\lambda>\Theta_0$, otherwise this would violate the stated Weyl laws above.
\end{remark}

\section{Model operators}\label{sec:pre}

\subsection{The de Gennes model}\label{sec:dG}

For $\xi\in\reals$, let $\mu_1(\xi)$ be the lowest eigenvalue of the operator
\[ -\partial_t^2+(t-\xi)^2\]
in $L^2(\reals_+)$, subject to the Neumann boundary condition at the origin, $\varphi'(0)=0$.

Let $\Theta_0=\inf_{\xi\in\reals}\mu_1(\xi)$, also called the de Gennes constant. It is a well-known fact that $\frac12<\Theta_0<1$ and with $\xi_0=\sqrt{\Theta_0}$, we have $\mu_1(\xi_0)=\Theta_0$ (see \cite[Section~3.2]{FH-b}). Also, $\mu_1(0)=1$ and $\lim_{\xi\to+\infty}\mu_1(\xi)=1$.

We denote by $\varphi_0$ the unique normalized ground state satisfying
\[ -\varphi _ 0 '' + (t-\xi_0)^2\varphi_0=\Theta_0\varphi_0,\quad \varphi_0>0, \quad \varphi_0'(0)=0.\]

\subsection{Family of Schr\"odinger operators}\label{sec:ms}

For $a,\xi\in\reals$, let $\mu_1(a,\xi)$ be the lowest eigenvalue of the Schr\"odinger operator
\[ -\partial_t^2+V_{a,\xi}(t)\]
in $L^2(\reals)$, where
\[
  V_{a,\xi}(t)=
    \begin{cases}
      (t - \xi) ^ 2  &\text{if } t > 0,\\
      (at - \xi) ^ 2 &\text{if } t < 0.
    \end{cases}
\]
We introduce the spectral parameter \cite{AKPS}
\[
  \beta_a =\inf_{\xi\in\reals}\mu_1(a,\xi)
\]
and we recall the following properties
\begin{itemize}
\item $\beta_a=\Theta_0$ for $a=-1$ and $\beta_a=a$ for $0\leq a\leq 1$;
\item For $-1<a<0$, there exists a unique $\xi_a>0$ such that $\beta_a=\mu_1(a,\xi_a)$; furthermore,
$|a|\Theta_0<\beta_a<\min(\Theta_0,|a|)$ (see \cite[Theorem~1.1]{AK} where in the notation of \cite{AK}, $\xi_a=-\zeta_a$);
\item The function $[-1,1]\ni a\to\beta_a$ is continuous.
\end{itemize}
We also introduce the lowest eigenvalue $\hat\mu_1(a,\xi)$ of the Schr\"odinger operator
\[ -\partial_t^2+\hat V_{a,\xi}(t)\]
in $L^2(\reals)$, where
\[
  \hat V_{a,\xi}(t)
  =
  \begin{cases}
    (at - \xi) ^ 2 & \text{if } t > 0,\\
    (t - \xi) ^ 2  & \text{if } t < 0.
  \end{cases}
\]
The corresponding spectral parameter is
 \begin{equation}\label{eq:def-hat-beta}
 \hat\beta_a =\inf_{\xi\in\reals}\hat\mu_1(a,\xi).
 \end{equation}
By the change of variable $\tau=|a|^{1/2}t$, it is easy to check that,
\[
  \hat\beta_a=|a|\beta_{\frac1{a}}\quad\text{ for }a\neq 0.
\]
Consequently we have
\[
  \begin{cases}
    \Theta_0 < \hat{\beta}_a < \min(|a|\Theta_0,1),
    & \text{ for $a<-1$,}\\
    \hat{\beta}_a = \Theta_0,
    & \text{ for $a=-1$,}\\
    \hat{\beta}_a = 1,
    & \text{ for $a\geq 1$.}\\
  \end{cases}
\]
Also the change of variable $\tau=-t$ yields
\[
  \hat\beta_a=\beta_a \quad\text { for }a\not=0.
\]
and we get
\[
 \begin{cases}
 |a|\Theta_0 < \hat\beta_a < \Theta_0 &\text{ for $-1<a<0$,}\\
 \hat\beta_a=a&\text{ for $0<a<1$.}
\end{cases}
\]
Finally, we determine the behavior of $\hat\beta_a$ as $a\to-\infty$.
\begin{proposition}\label{prop:beta-a}
As $a\to-\infty$, $\hat\beta_a\to 1$.
\end{proposition}
\begin{proof}
Since $\hat\beta_a=|a|\beta_{\frac1a}$ for $a<0$, it suffices to show that
\begin{equation}\label{eq:beta-a}
  \beta_a= |a|+o(|a|)\quad \text{as } a \to 0^-.
\end{equation}
We know that by \cite[Theorem~2.6 and Proposition A.6]{AKPS}
\[
  |a|\Theta_0<\beta_a < |a|
  \quad
  \text{for $-1 < a < 0$},
\]
hence $\beta_a\to0$ as $a\to0^-$. Passing to a subsequence, suppose that
$\xi_a/|a|\to c_*$ as $a\to0^-$, with $c_*\in[0,+\infty]$. We will prove that
$c_*=+\infty$.

Recall that, by \cite[Theorem~1.1 and Eq. (27)]{AK},
\[
  \xi_a = \sqrt{\beta_a+\gamma_a^2},
  \quad
  \gamma_a = \phi_a'(0)/\phi_a(0),
  \quad \beta_a\geq \hat\lambda(\gamma_a,\xi_a),
\]
where $\phi_a$ is a normalized and positive ground state of $\beta_a$, and
$\hat\lambda(\gamma,\xi)$ is the lowest eigenvalue of
\[
  -\partial_t^2+(t-\xi)^2
\]
in $L^2(\reals_+)$ with the boundary condition $\varphi'(0)=\gamma\varphi(0)$.
Note that $\hat\lambda(0,\xi)=\mu_1(\xi)$ is the lowest eigenvalue of the
de Gennes model in Appendix~\ref{sec:dG}.

If $\xi_a\to 0$, then $\gamma_a\to0$ and
$\hat\lambda(\gamma_a,\xi_a)\to\lambda(0,0)=\mu_1(0)=1$, which violates
$\beta_a\geq \hat\lambda(\gamma_a,\xi_a)$. Therefore, $c_*$ cannot be finite.

From \cite[Lemma~A.3]{AKPS},
\[
  \beta_a\geq\min\Bigl(|a|\mu_1(\xi_a/|a| ),\mu_1(\xi_a)\Bigr)\geq \min(|a|\mu_1(\xi_a/|a| ),\Theta_0 )\quad(-1<a<0),
\]
and knowing that $\xi_a/|a|\to+\infty$, we get $\mu_1(\xi_a/|a| )\to1$ and
\[
  \beta_a\geq |a|+o(|a|)\quad\text{as }a \to 0^-.
\]
\end{proof}

\subsection*{Acknowledgments}

Part of this work was carried out during a visit of A. Kachmar at LMU Munich and Lund University in March
2025, supported by CUHKSZ under grant No. UDF01003322 and by the Deutsche Forschungsgemeinschaft (DFG, German Research Foundation) via TRR
352 – Project-ID 470903074. A. Kachmar was partially
supported by a startup fund at AUB (grant no. 513125). E. L. Giacomelli was partially funded by the Deutsche Forschungsgemeinschaft (DFG, German Research Foundation) via TRR
352 – Project-ID 470903074 and supported by Gruppo Nazionale per la Fisica Matematica in Italy.

\end{document}